\documentclass[a4paper]{article}
\usepackage[all]{xy}\usepackage[latin1]{inputenc}        
\usepackage[dvips]{graphics,graphicx}
\usepackage{amsfonts,amssymb,amsmath,color,mathrsfs, amstext}
\usepackage{amsbsy, amsopn, amscd, amsxtra, amsthm,authblk}
\usepackage{enumerate,algorithmicx,algorithm}
\usepackage{algpseudocode}
\usepackage{upref}
\usepackage{geometry}
\usepackage{amsthm,amsmath,amssymb}
\usepackage{relsize}
\usepackage{mathrsfs}
\usepackage{bm}
\usepackage{exscale}
\usepackage{graphicx}
\usepackage{tabularx}
\usepackage{threeparttable,multirow,dcolumn,booktabs}
\usepackage{algorithmicx,algorithm}
\usepackage{tabularx}
\usepackage{caption}
\usepackage{float}
\usepackage{bm}
\usepackage{caption}
\usepackage{yhmath}
\usepackage{booktabs}
\usepackage{subcaption}
\usepackage{multirow}
\usepackage{makecell}
\usepackage[colorlinks,
            linkcolor=red,
            anchorcolor=red,
            citecolor=red
            ]{hyperref}

\numberwithin{equation}{section}

\newcommand{\V}[1]{\boldsymbol{#1}} 
\newcommand{\M}[1]{\boldsymbol{#1}} 
\newcommand{\abs}[1]{\left|#1\right|} 

\newcommand{\grad}{\M{\nabla}} 

\newcommand{\sM}[1]{\M{\mathcal{#1}}} 
\newcommand{\eps}{\epsilon}
\newcommand{\erf}{\mathrm{erf}}
\newcommand{\erfc}{\mathrm{erfc}}
\newcommand{\m}[1]{\mathrm{#1}}

\usepackage{soul}
\usepackage{enumitem}
\usepackage{bbm}
\newtheorem{thm}{Theorem}[section]
\newtheorem{lem}[thm]{Lemma}
\newtheorem{defi}[thm]{Definition}
\newtheorem{rmk}[thm]{Remark}
\newtheorem{prop}[thm]{Proposition}
\usepackage{comment}
\usepackage{exscale}
\usepackage{relsize}
\usepackage{booktabs}
\usepackage{longtable}
\usepackage{multirow}
\usepackage{makecell}

\usepackage{hyperref}
\usepackage{algorithm}
\usepackage{algpseudocode}
\usepackage{setspace}

\begin{document}

\title{Fast Algorithm for Quasi-2D Coulomb Systems}

\author[1,2]{Zecheng Gan\thanks{zechenggan@ust.hk}}
\author[1,3]{Xuanzhao Gao\thanks{xz.gao@connect.ust.hk}}
\author[4,5]{Jiuyang Liang\thanks{liangjiuyang@sjtu.edu.cn}}
\author[4]{Zhenli Xu\thanks{xuzl@sjtu.edu.cn}}

\affil[1]{Thrust of Advanced Materials, The Hong Kong University of Science and Technology (Guangzhou), Guangdong, China}
             
\affil[2]{Department of Mathematics, The Hong Kong University of Science and Technology, Clear Water Bay, Kowloon, Hong Kong SAR} 

\affil[3]{Department of Physics, The Hong Kong University of Science and Technology, Clear Water Bay, Kowloon, Hong Kong SAR}

\affil[4]{School of Mathematical Sciences, MOE-LSC and CMA-Shanghai, Shanghai Jiao Tong University, Shanghai, China}

\affil[5]{Center for Computational Mathematics, Flatiron Institute, Simons Foundation, New York, USA}
            


\date{}
\maketitle

\begin{abstract}
Quasi-2D Coulomb systems are of fundamental importance and have attracted much attention in many areas nowadays.
Their reduced symmetry gives rise to interesting collective behaviors, but also brings great
challenges for particle-based simulations. 
Here, we propose a novel algorithm framework to address the $\mathcal O(N^2)$ simulation complexity associated with the long-range nature of Coulomb interactions. 
First, we introduce an efficient Sum-of-Exponentials (SOE) approximation for the long-range kernel associated with Ewald splitting, achieving uniform convergence in terms of inter-particle distance, which reduces the complexity to $\mathcal{O}(N^{7/5})$. 
We then introduce a random batch sampling method in the periodic dimensions, the stochastic approximation is proven to be both unbiased and with reduced variance via a tailored importance sampling strategy, further reducing the computational cost to $\mathcal{O}(N)$.
The performance of our algorithm is demonstrated via various numerical examples. Notably, it achieves a speedup of $2\sim 3$ orders of magnitude comparing with Ewald2D method, enabling molecular dynamics (MD) simulations with up to $10^6$ particles on a single core. 
The present approach is therefore well-suited for large-scale particle-based simulations of Coulomb systems under confinement, making it possible to investigate the role of Coulomb interaction in many practical situations.
\end{abstract}








\section{Introduction} \label{sec:intro}

In various fields such as electromagnetics, fluid dynamics, computational soft matter and materials science~\cite{anderson1995computational,frenkel2023understanding,liu20192d}, it is of great importance to evaluate lattice kernel summations in the form of
\begin{equation}\label{eq::phi12}
	\phi(\V x) = \sum_{\V{m}} \sum_{j = 1}^{N} \rho_j K(\V{x} - \V{y}_j + \V{m} \circ \V{L})\;,
\end{equation}
where~$\V{x}, \V{y}_j \in \mathbb{R}^d$ are $d$-dimensional vectors in a rectangular box $\Omega$ with $\V{L}$ the vector of its edge lengths, $\rho_j$ refers to the density or weight, $\V{m} \in \mathbb{Z}^{d^\prime} \otimes \{0\}^{d-d^\prime}$ exerts periodicity in the first $d^{\prime}$ directions (with $d^{\prime}\leq d$), ``$\circ$'' represents the Hadamard product, and~$K(\V{x})$ is the kernel function whose form depends on the interested physical problem. 
If~$d^\prime = d$, the system is called fully-periodic, $d^\prime=0$, it is in a free-space, otherwise it is called partially-periodic. 
In this work, we focus on the doubly-periodic case (where $d = 3$ and $d^\prime = 2$), characterizing a confined system with nanometer/angstrom length scale in one direction; bulk and periodic in the other two directions. In literature, this type of systems are also referred as quasi-2D systems, which have caught much attention in studies of magnetic and liquid crystal films, super-capacitors, crystal phase transitions, dusty plasmas, ion channels, superconductive materials and quantum devices~\cite{kawamoto2002history, hille2001ionic, teng2003microscopic, Messina2017PRL, mazars2011long, saito2016highly, liu20192d}. 

The reduced symmetry of quasi-2D systems gives rise to new phenomena, but also brings formidable challenges in both theory and computation.
The first challenge comes from the involved \emph{long-range} interaction kernels, including but not limited to Coulomb and dipolar kernels in electrostatics, Oseen and Rotne-Prager-Yamakawa kernels in hydrodynamics and the static exchange-correlation kernels in density functional theory calculations. 
For fully-periodic or free-space systems, $\mathcal O(N)$ fast algorithms have been developed; but the field is still under developing for partially-periodic systems.
The anisotropy of such systems poses extra challenges for simulations:
1) the periodic and non-periodic directions need to be handled separately due to their different boundary conditions and length scales;
2) the convergence properties of the lattice kernel summation Eq.~\eqref{eq::phi12} requires careful consideration, which largely depend on the well-poseness of the underlying PDEs. 
Another challenge comes from practical applications.
To accurately determine the phase diagram of a many-body system may require thousands of simulation runs under different conditions~\cite{levin2002electrostatic}, each with billions of time steps to sample ensemble averages.
Moreover, to eliminate the finite size effect, millions of free particles need to be simulated. Such large-scale simulations are especially required for quasi-2D systems, so as to accommodate its strong anisotropy, and resolving possible boundary layers forming near the confinement surfaces~\cite{mazars2011long}. 
The cumulative impact of these considerations poses significant challenges for numerical simulations for quasi-2D systems. 

To address these issues associated with the particle-based simulation of quasi-2D systems, a variety of numerical methods have been developed.
Most of them fall into two categories: 
(1) Fourier spectral methods~\cite{lindbo2012fast,nestler2015fast,doi:10.1021/acs.jctc.3c01124, maxian2021fast}, where particles are first smeared onto grids, and subsequently the underlying PDE is solved in Fourier domain where fast Fourier transform (FFT) can be used for acceleration; 
(2) adaptive tree-based methods, where fast multipole method (FMM)~\cite{greengard1987fast} or tree code~\cite{Barnes1986Nature} orginally proposed for free-space systems can be extended to quasi-2D systems by careful extension to match the partially-periodic boundary conditions~\cite{yan2018flexibly,liang2020harmonic}. 
Alternative methods have also been proposed, such as the Lekner summation-based MMM2D method~\cite{arnold2002novel}, multilevel summation methods~\cite{doi:10.1021/ct5009075,greengard2023dual}, and correction-based approaches such as Ewald3DC~\cite{yeh1999ewald} and EwaldELC~\cite{arnold2002electrostatics}, which first solve a fully-periodic system and then add the partially-periodic correction terms. 
By combining with either FFT or FMM, these methods achieve $\mathcal{O}(N\log N)$ or even $\mathcal{O}(N)$ complexity. 

However, the issue of large-scale simulation of quasi-2D systems is still far from settled.
A few challenges remains. First, FFT-based methods need extra techniques to properly handle the non-periodic direction, such as truncation~\cite{parry1975electrostatic}, regularization~\cite{nestler2015fast}, or periodic extension~\cite{lindbo2012fast}, which may lead to algebraic convergence or require extra zero-padding to guarantee accuracy. 
Recent advancements by Shamshirgar \emph{et al.}~\cite{shamshirgar2021fast}, combining spectral solvers with kernel truncation methods (TKM)~\cite{vico2016fast}, have reduced the zero-padding factor from $6$ to $2$~\cite{lindbo2012fast}, which still requires doubling the number of grids with zero-padding. 
Second, the periodization of FMM needs to encompass more near-field contributions from surrounding cells~\cite{yan2018flexibly,barnett2018unified}. The recently proposed 2D-periodic FMM~\cite{PEI2023111792} may offer a promising avenue; however, it has not yet been extended to partially-periodic problems.
Finally, it is worth noting that most of the aforementioned issues will become more serious when $L_z\ll \min\{L_x, L_y\}$, in which case the Ewald series summation will converge much slower~\cite{arnold2002electrostatics}, and the zero-padding issue of FFT-based methods also becomes worse~\cite{maxian2021fast}.

In this work, we introduce a novel algorithm for particle-based simulations of quasi-2D systems with long-range interactions (typically, the $1/r$ Coulomb kernel). 
Our approach connects Ewald splitting with a sum-of-exponentials (SOE) approximation, which ensures uniform convergence along the whole non-periodic dimension. 
We further incorporate importance sampling in Fourier space over the periodic dimensions, achieving an overall $\mathcal{O}(N)$ simulation complexity. 
The algorithm has distinct features over the other existing approaches:

\begin{itemize}
	\item[1.] Our approach provides a well-defined computational model for both discrete free-ions and continuous surface charge densities, and is consistent under both NVT and NPT ensembles.
	\item[2.] The simulation algorithm has linear complexity with small prefactor, and it does not depend on either FFT or FMM for its asymptotic complexity.
	\item[3.] Instead of modifying FFT and FMM-based methods, which are originally proposed for periodic/free-space systems, our scheme is tailored for partially-periodic systems, it perfectly handles the anisotropy of such systems without any loss of efficiency.
	\item[4.] Our method is mesh free, and can be flexibly extended to other partially-periodic lattice kernel summations in arbitrary dimensions, thanks to the SOE approximation and random batch sampling method.
\end{itemize}

Our approach builds upon the Ewald2D formula and incorporates an SOE approximation for the kernel function in the non-periodic dimension. By utilizing the SOE form, we are able to reformulate Ewald2D into a recursive summation, reducing the computational complexity from $\mathcal{O}(N^2)$ to $\mathcal{O}(N^{7/5})$. 
We also address the issue of catastrophic error cancellation associated with the original Ewald2D method. 
Additionally, we introduce a random batch importance sampling technique in Fourier space to accelerate the computation in the periodic dimensions, without the need for costly direct summation or FFT. 
The resulting method, named RBSE2D, maintains numerical stability and achieves optimal $\mathcal{O}(N)$ complexity in both CPU and memory consumptions. 
Rigorous error estimates and complexity analysis are provided, further validated by numerical tests. 
In particular, numerical results demonstrate that RBSE2D-based MD simuations can accurately reproduce the spatiotemporal properties of quasi-2D Coulomb systems, along with a significant improvement in computational efficiency with a speedup of approximately $2-3$ orders of magnitude compared to the standard Ewald2D method, allowing large-scale simulations of quasi-2D systems.

We refer to the RBSE2D as a \emph{framework} for partially-periodic summation problems with arbitrary non-oscillatory kernels since the method is highly kernel/dimension independent.
The details of this framework, however, are showcased by the specific Coulomb kernel under the quasi-2D setup, which is both physically important and concise to be mathematically clarified.
The remaining sections of the paper are organized as follows. 
Section~\ref{sec:overview} introduces the quasi-2D electrostatic model and revisits the Ewald2D summation formula for quasi-2D Coulomb systems. Section~\ref{sec:method} introduces the SOE approximation for the Ewald2D summation in the non-periodic dimension. Section~\ref{sec:rbm} introduces the random batch sampling method for further accelerating the comptations in periodic dimensions.
Finally, to validate the accuracy and efficiency of our proposed method, numerical results are presented in Section~\ref{sec:md}. 
Concluding remarks are provided in Section~\ref{sec:conclusion}.

\section{Ewald summation for quasi-2D Coulomb systems} \label{sec:overview}
In this section, we introduce the physical model and mathematical notations for quasi-2D Coulomb systems, and provide a concise overview of the Ewald2D lattice summation formula, along with the extension to confinement with charged interfaces. While some of the results in this section are novel contributions, it is important to note that the majority of concepts introduced in Sections~\ref{eq::quasi2DCoulomb}-\ref{subsec::elec} have been well-established in the literature.

\subsection{Quasi-2D Coulomb systems}\label{eq::quasi2DCoulomb}
Quasi-2D Coulomb systems are usually modelled  via the so-called doubly-periodic boundary conditions (DPBCs), i.e., periodic in $xy$ directions to mimic the environment of bulk, and non-periodic in the $z$ direction, indicating confined particles in $z$ at nano-/micro scales either by soft or hard potential constraints~\cite{mazars2011long}.
In literature, this type of model is often referred as the ``slab geometry''~\cite{yeh1999ewald} or the ``slit channel'' geometry~\cite{maxian2021fast}. 

Consider a simulation domain $\Omega=[0,L_x]\times[0,L_y]\times[0,L_z]\subset\mathbb{R}^3$, which comprises~$N$ particles with positions~$\V{r}_{i}=(x_{i},y_{i},z_{i})\in\Omega$ and charges~$q_{i}$, $i=1,\cdots,N$.
The electrostatic potential $\phi(\V r)$ for such systems, assuming uniform background dielectric media, is governed by the following Poisson's equation with DPBCs:
\begin{equation}\label{eq::Poisson}
	-\Delta\phi(\bm{r}) = 4\pi g(\bm{r}), \quad \text{with} \quad g(\bm{r})= \sum_{\V{m}}\sum_{j=1}^{N}q_{j}\delta(\bm{r}-\bm{r}_{j}+\bm{\mathcal{M}}),
\end{equation}
where $\bm{m}=(m_x,m_y) \in\mathbb{Z}^2$, and $\bm{\mathcal{M}} := (m_x L_x, m_y L_y, 0)$.
The solution to Eq.~\eqref{eq::Poisson} is doubly-periodic $\phi(\bm{r})=\phi(\bm{r} + \bm{\mathcal{M}})$ and unique up to a linear function in $z$. The uniqueness will be satisfied by incorporating a suitable boundary condition as $z\to\pm\infty$.

In many practical situations, the potential is defined via the following Coulomb summation formulation,
\begin{equation}\label{eq::pairssum}
	\phi(\bm{r})=\sum_{\V{m}}\sum_{j=1}^{N}\frac{q_{j}}{|\bm{r}-\bm{r}_{j} + \bm{\mathcal{M}}|}\;.
\end{equation}
It is important to note that the potential becomes singular when $\bm{r}=\bm{r}_{j}$ and $\bm{m}=\bm{0}$, with the singularity arising from the Dirac delta source. 
It is important to note that Eq.~\eqref{eq::pairssum} is not well defined without specifying the shape of summation region~\cite{de1980simulation,smith2008electrostatic} and the total charge neutrality condition. 
We construct copies $\Omega(\bm{m})$ of the simulation domain by $\Omega(\bm{m})=\{\bm{r}|\bm{r}-\bm{\mathcal{M}}\in\Omega\}$. One has $\Omega(\bm{0})\equiv\Omega$ and a copied domain $\Omega(\bm{m})$ contains $N$ charges at $\bm{r}_j+\bm{\mathcal{M}}$. 
Next, we define a summation shape $\mathcal{S} \subset \mathbb{R}^2$ which contains the origin.
We consider a lattice $\Lambda(\mathcal{S},R)=\{\Omega(\bm{m})|\bm{m}/R\in\mathcal{S}\}$ with $R\in\mathbb{R}$ be a truncation parameter. 
Given these notations, Theorem \ref{order} clarifies the necessary conditions to guarantee absolute convergence of the series.

\begin{thm} \label{order}
The summation in Eq.~\eqref{eq::pairssum} truncated within region $\Lambda(\mathcal{S},R)$ is absolutely convergent as~$R \to \infty$ if (1) the shape $\mathcal{S}$ is symmetric around the origin (meaning that if $\bm{m} / R \in \mathcal{S}$, then $- \bm{m}/ R \in \mathcal{S}$);and (2) the system within the central box is charge neutral, i.e.,~$\sum_{j = 1}^N q_j = 0$.
\end{thm}



\begin{proof}
	By the Taylor expansion, one has for large $|\bm{m}|$
	\begin{equation}\label{eq::3}
		\frac{1}{|\bm{r} + \bm{\mathcal{M}}|} = \frac{1}{|\bm{\mathcal{M}}|} - \frac{\bm{r} \cdot \bm{\mathcal{M}}}{|\bm{\mathcal{M}}|^3} + \mathcal{O} \left(\frac{1}{|\mathcal{\sM M}|^3}\right)\;.
	\end{equation}
	In the right-hand side of Eq.~\eqref{eq::3}, the second term is odd with respect to $\V{m}$ and thus sums to zero due to the symmetry of $\mathcal{S}$. The last $\mathcal{O}(|\V{\mathcal{M}}|^{-3})$ term is absolutely convergent as $R\rightarrow \infty$. Therefore, it is sufficient to analyze the convergence behavior of the expression:
	\begin{equation}\label{eq::4}
        J=\lim_{R\rightarrow\infty}\sum_{\bm{m}/ R \in \mathcal{S}}\frac{1}{|\bm{\mathcal{M}}|}\sum_{j=1}^{N}q_{j}.
	\end{equation}
	Eq.~\eqref{eq::4} can be viewed as a Riemann sum multiplied by the total net charges. If the charge neutrality condition is satisfied, i.e., $\sum_{j=1}^{N}q_{j}=0$, then $J$ vanishes and the series summation of $\phi$ in Eq.~\eqref{eq::pairssum} is absolutely convergent. 
	If the charge neutrality condition is violated, the Riemann sum can be approximated as an integral, $J\sim 2\pi (L_xL_y)^{-1} R\sum_{j=1}^{N}q_{j}$, which diverges as $R\rightarrow \infty$. This implies that the total charge neutrality condition is a necessary requirement for the existence of $\phi(\bm{r})$ in Eq.~\eqref{eq::pairssum}.
\end{proof}

In practice, a common choice in Ewald 3D/2D summation approaches is by choosing the spherical/circular shape of summation region centered at the origin with unit radius; i.e., the sum is taken over $\abs{\V m}=0, 1, 2\ldots, R$ in ascending order, where $R$ is the truncation parameter.
As long as Eq.~\eqref{eq::pairssum} is well defined, Proposition~\ref{prop::boundary} establishes a precise relationship between Eq.~\eqref{eq::pairssum} and the solution to Poisson's equation Eq.~\eqref{eq::Poisson} with a properly chosen boundary condition as $z\to\infty$.

\begin{prop}\label{prop::boundary}
	If the the series summation of $\phi(\bm{r})$ in Eq.~\eqref{eq::pairssum} satisfies both conditions stated in Theorem~\ref{order}, then it is a unique solution to Poisson's equation Eq.~\eqref{eq::Poisson} given the far-field boundary condition
	\begin{equation}\label{eq::boundary1}
		\lim_{z \to \pm \infty} \phi(\bm{r}) = \pm \frac{2\pi}{L_x L_y} \sum_{j=1}^{N} q_{j} z_{j}\;.
	\end{equation}
\end{prop}

\begin{proof}
	Let~$\bm{\rho}=(x,y)$ and $\bm{k}=(k_x, k_y)$ denote the periodic dimensions of position and Fourier frequency, respectively, where $\bm{\rho}\in\mathcal{R}^2$ and $\bm{k}\in\mathcal{K}^2$ with
	\begin{equation}
		\mathcal{R}^2:=\{\bm{\rho}\in[0,L_x]\times [0,L_y]\},\quad \text{and}\quad \mathcal{K}^2:=\left\{\bm{k}\in\frac{2\pi}{L_x}\mathbb{Z}\times \frac{2\pi}{L_y}\mathbb{Z}\right\}\;.
	\end{equation} 
	The Poisson's summation formula~(see~\ref{app::Fourier}) indicates 
	\begin{equation}\label{eq::9}
		\sum_{\V{m}} \sum_{j = 1}^N \frac{q_{j}}{|\bm{r}-\bm{r}_{j} + \bm{\mathcal{M}}|} =  \sum_{j = 1}^N q_j \left[ \frac{2 \pi}{L_x L_y} \sum_{\bm{k} \neq \bm{0}} \frac{e^{-k \abs{z - z_j}}}{k} e^{- \m{i} \bm{k} \cdot (\bm{\rho} - \bm{\rho_j})} + \phi_{\bm{0}}(z-z_j)\right]\;,
	\end{equation}
	where~$k=|\bm{k}|$ and 
	\begin{equation}\label{eq::91}
		\phi_{\bm{0}}(z-z_j) = \frac{2\pi}{L_x L_y} \int_{0}^\infty \frac{ \rho}{\sqrt{\rho^2 + |z - z_j|^2}} d\rho
	\end{equation}
	represents the~$\bm{k}=\bm{0}$ term. Note that Eq.~\eqref{eq::91} is equivalent to a uniformly charged infinite plane in the real space. 
	As~$z \to \infty$, all~$\bm{k} \neq \bm{0}$ modes vanish, so that 
	\begin{equation}
		\lim_{z \to \pm \infty} \phi(\bm{r})=\lim_{z \to \pm \infty}\sum_{j=1}^{N}q_j\phi_{\bm{0}}(z-z_j).
	\end{equation}
	One can then integrate out Eq.~\eqref{eq::91} and arrives at
	\begin{equation}\label{eq::boundary2}
		\lim_{z \to \pm \infty} \phi(\bm{r}) = -\lim_{z \to \pm \infty} \frac{2\pi}{L_x L_y} \sum_{j = 1}^N q_j \abs{z - z_j}.
	\end{equation}
	Finally, the charge neutrality condition results in Eq.~\eqref{eq::boundary1}. 
 
Eq.~\eqref{eq::boundary1} indicates that $\lim\limits_{z \to \pm \infty} \phi(\bm{r})$ is a finite constant, and thus can be regarded as a properly chosen Dirichlet-type boundary condition at $z\rightarrow\pm\infty$~\cite{lindbo2012fast}. 
Next, we study the uniqueness of the solution of Poisson's equation under DPBCs and Eq.~\eqref{eq::boundary1}. Suppose there exists two solutions $\phi_1(\bm{r})$ and $\phi_2(\bm{r})$,
let $u(\bm{r}):=\phi_1(\bm{r})-\phi_2(\bm{r})$ be the difference between two solutions, and $\mathcal{B}=\mathcal{R}^2\times\mathbb{R}$ be a tubular cell that extends to infinity in the $z$-direction. By Green's first identity, we have 
\begin{equation}
0=\int_{\mathcal{B}}u\Delta ud\bm{r}=\int_{\mathcal{B}}\nabla\cdot(u\nabla u)-(\nabla u)^2d\bm{r}=\int_{\partial\mathcal{B}}u\nabla u\cdot d\bm{S}-\int_{\mathcal{B}}\left(\nabla u\right)^2d\bm{r}.
\end{equation}
The boundary term in the RHS cancels by periodicity in the $xy$-plane as well as $\lim\limits_{z\rightarrow\pm\infty}u(\bm{r})=0$, hence $\nabla u(\bm{r})\equiv \bm{0}$ in $\mathcal{B}$. Accordingly, we have $u(\bm{r})\equiv 0$, which ensures the uniqueness of the solution.
  
\end{proof}

For such a well-defined quasi-2D Coulomb system, the electrostatic \emph{interaction energy} $U$ is given by
\begin{equation} \label{eq::U_pair}
	U(\bm{r}_1, \ldots,\bm{r}_N)= \frac{1}{2} \sum_{\V{m}} \sum_{i=1}^{N} \sum_{j=1}^N {}^\prime \frac{q_i q_j} {|\bm{r}_{i}-\bm{r}_{j} + \bm{\mathcal{M}}|}\;,
\end{equation}
where the notation ``$\prime$'' represents that the~$i = j$ case is excluded when~$\bm{m} = \bm{0}$.
The corresponding force on each particle is $\bm{F}^i=-\nabla_{\bm{r}_{i}} U$, for $i=1,2,\ldots, N$.
It is remarked that, though the quasi-2D Coulomb summation is absolutely convergent, due to the long-range nature of Coulomb interaction, directly truncating the series for computing energy or force will lead to slow convergence with a complexity of $\mathcal O(N^2)$.

\subsection{Ewald2D summation revisited}\label{subsec::elec}

Throughout the remainder sections, we will extensively use Fourier transforms for the DPBCs, which leads to the well-known Ewald2D~\cite{parry1975electrostatic, heyes1977molecular, de1979electrostatic} formula for quasi-2D Coulomb systems.
For ease of discussion, the mathematical notations and definitions are first provided.

\begin{defi}\label{Def::Fourier}(Quasi-2D Fourier transform) 
	Let $f(\bm{\rho},z)$ be a function that is doubly-periodic in $xy$-dimensions, its quasi-2D Fourier transform is defined by
	\begin{equation}
		\widetilde{f}(\bm{k},\kappa):=\int_{\mathcal{R}^2}\int_{\mathbb{R}}f(\bm{\rho},z)e^{-\m{i} \bm{k}\cdot\bm{\rho}}e^{- \m{i} \kappa z}dzd\bm{\rho}.
	\end{equation}
	The function $f(\bm{\rho},z)$ can be recovered from the corresponding inverse quasi-2D Fourier transform:
	\begin{equation}
		f(\bm{\rho},z)=\frac{1}{2 \pi L_x L_y}\sum_{\bm{k} \in \mathcal{K}^2} \int_{\mathbb{R}} \widetilde{f}(\bm{k}, \kappa) e^{\m{i} \bm{k} \cdot \bm{\rho}} e^{\m{i} \kappa z} d\kappa\;.
	\end{equation}
\end{defi}

In order to calculate Eq.~\eqref{eq::pairssum}, the Ewald splitting based methods~\cite{Ewald1921AnnPhys} are often adopted. 
The idea of the Ewald splitting technique can be understood as decomposing the source term $g(\bm{r})$ of Eq.~\eqref{eq::Poisson} into the sum of short-range and long-range components:
\begin{equation}
	g(\bm{r})=\left[g(\bm{r})-(g\ast\tau)(\bm{r})\right]+(g\ast\tau)(\bm{r}):=g_{s}(\bm{r})+g_{\ell}(\bm{r}),
\end{equation}
where the symbol ``$\ast$'' denotes the convolution operator defined in Eq.~\eqref{eq:Q2D_cov}, and $\tau(\bm{r})$ is the screening function.
In the standard Ewald splitting~\cite{Ewald1921AnnPhys,tornberg2016ewald} for quasi-2D systems, $\tau$ is chosen to be a Gaussian function that is periodized in the $xy$-plane, hence $\widetilde{\tau}$ is also a Gaussian, 
\begin{equation}  
\tau(\bm{\rho},z)=\sum_{\bm{m}}\pi^{-3/2}\alpha^3 e^{-\alpha^2 |\bm{r}+\bm{\mathcal{M}}|^2},\quad~ \widetilde{\tau}(\bm{k},\kappa)=e^{-(k^2+\kappa^2)/(4\alpha^2)},
\end{equation}
where $\alpha>0$ is a parameter to be optimized for balancing the computational cost in short-range and long-range components.
The electrostatic potential at the $i$th particle location can be expressed as 
\begin{equation}\label{eq::phi}
	\phi(\bm{r}_{i}) :=\phi_{s}(\bm{r}_{i}) + \phi_{\ell}(\bm{r}_{i}) - \phi_{\text{self}}^{i}\;,
\end{equation}           
where the short-range ($\phi_{s}$) and long-range ($\phi_{\ell}$) components are given as:
\begin{align}
	\phi_{s}(\bm{r}_{i}) &= \sum_{\V{m}} \sum_{j=1}^N {}^\prime \frac{q_{j} \erfc (\alpha \abs{\V{r}_{ij} + \V{\mathcal{M}}})}{\abs{\V{r}_{ij} + \V{\mathcal{M}}}}, \label{eq:phi_s} \\
	\phi_{\ell}(\bm{r}_{i}) &= \sum_{\V{m}} \sum_{j=1}^N  \frac{q_{j} \erf (\alpha \abs{\V{r}_{ij} + \V{\mathcal{M}}})}{\abs{\V{r}_{ij} + \V{\mathcal{M}}}},\label{eq:phi_l}
\end{align}
with $\bm{r}_{ij}:=\bm{r}_{i}-\bm{r}_{j}$ and the error function $\erf(\cdot)$ and complementary error function $\erfc(\cdot)$ defined as
\begin{equation}
	\erf(x) := \frac{2}{\sqrt{\pi}} \int_0^{x} e^{-t^2} dt~\quad\text{and~}\erfc(x) := 1 - \erf(x), \label{eq:erf}
\end{equation}
respectively. 
In Eq.~\eqref{eq:phi_s}, $\sum^\prime$ indicates that the sum excludes the self interaction term when $j=i$ and $\V{m}=\V{0}$; and in Eq.~\eqref{eq::phi}, $\phi_{\text{self}}^{i}$ is the unwanted interaction between the Gaussian and point source, which should also be subtracted for consistency,
\begin{equation}\label{eq::self}
	\phi_{\text{self}}^{i}=\lim_{r\rightarrow 0}\frac{q_{i} \erf (\alpha r)}{r}=\frac{2\alpha }{\sqrt{\pi}}q_{i},
\end{equation}
where $r=\sqrt{\rho^2+z^2}$ and $\rho=|\bm{\rho}|$. It is clear that $\phi_{s}$ converges absolutely and rapidly due to the Gaussian screening, one can efficiently evaluate it in real space by simple truncation. 
Conversely, $\phi_{\ell}$ is still slowly decaying in real space but the interaction becomes smooth -- the singularity of $1/r$ as $r\rightarrow 0$ is removed, making  $\phi_{\ell}$ fast convergent in the Fourier space. 
The detailed formulation for the 2D Fourier expansion of $\phi_{\ell}$ is provided below.

\begin{lem}\label{thm::SpectralExpansion}
	By Fourier transform in the periodic $xy$ dimensions, $\phi_{\ell}$ can be written as the following series summation in $k$-space:
	\begin{align}\label{eq::7} 
		\phi_{\ell}(\bm{r}_{i}) = \sum_{\V{k} \neq \bm{0}} \phi_{\ell}^{\V{k}}(\bm{r}_{i}) + \phi_{\ell}^{\bm{0}}(\bm{r}_{i})\;,
	\end{align}
	where the non-zero modes read
	\begin{align}\label{eq:philk}
		\phi_{\ell}^{\V{k}}(\bm{r}_{i}) &= \frac{\pi}{L_x L_y} \sum_{j = 1}^N q_{j} \frac{e^{\m{i} \V{k} \cdot \V{\rho}_{ij}}}{k} \left[\xi^{+}(k, z_{ij})+\xi^{-}(k, z_{ij})\right],
	\end{align}
	with $\bm{\rho}_{ij}=(x_{i}-x_{j},y_{i}-y_{j})$, $z_{ij}=|z_{i}-z_{j}|$, and
	\begin{equation}\label{eq::9}
		\xi^{\pm}(k, z_{ij}) := e^{\pm k z_{ij}} \erfc \left( \frac{k}{2 \alpha} \pm \alpha z_{ij}\right)\;,
	\end{equation} 
	and the 0-th mode is
	\begin{align}\label{eq:phil0}
		\phi_{\ell}^{\bm{0}}(\bm{r}_{i}) &= -\frac{2\pi}{L_x L_y} \sum_{j = 1}^N q_{j} \left[ {z_{ij}} \erf(\alpha {z_{ij}}) + \frac{e^{-(\alpha z_{ij})^2}}{\alpha \sqrt{\pi}}  \right]\;.
	\end{align}
\end{lem}

Eqs.~\eqref{eq:phi_s}, \eqref{eq::self} and \eqref{eq::7} constitute the well-known Ewald2D summation, which has been derived through various methods~\cite{parry1975electrostatic,tornberg2016ewald,heyes1977molecular,de1979electrostatic,PhysRevB.61.6706}.
An alternative derivation is provided in ~\ref{app::deriv}.


The Ewald2D summation is the exact solution and does not involve any uncontrolled approximation. 
However, two significant drawbacks limit its application for large-scale simulations:
\begin{itemize}
	\item Even with optimal choice of parameter $\alpha$, computing the interaction energy $U$ for an $N$-particle system through Eqs.~\eqref{eq:phi_s}, \eqref{eq::self} and \eqref{eq::7} takes $\mathcal O(N^2)$ complexity, which is worse than $\mathcal O(N^{3/2})$ for that of the Ewald3D, the fully-periodic case.
	\item The function $\xi^{\pm}(k, z_{ij})$ is ill-conditioned: It grows exponentially as $kz_{ij}$ grows, 
	leading to catastrophic error cancellation in actual computations with prescribed machine precision.
\end{itemize}

In this work, we develop an algorithm framework to address these two issues.
As will be shown in Section~\ref{sec:method}, we introduce the SOE approximation and a forward recursive approach, which reduce the computational complexity from $\mathcal O(N^2)$ to $\mathcal O(N^{7/5})$ without losing accuracy, while the ill-conditioning issue is also properly handled.
We further introduce a random batch importance sampling technique, outlined in Section~\ref{sec:rbm}, yielding an optimal complexity of $\mathcal O(N)$, allowing large-scale simulations of quasi-2D Coulomb systems.

\subsection{Error estimates for the Ewald2D summation}
Although the Ewald2D method is a widely recognized, standard technique, its theoretical error analysis remains underdeveloped.
In this section, we provide a truncation error analysis for the Ewald2D summation.
The truncation error is clearly configuration dependent.
Here the estimation is analyzed based on the \emph{ideal-gas assumption}~\cite{hansen2013theory}, which was used by Kolafa and Perram~\cite{kolafa1992cutoff} in analyzing the Ewald3D case. 
Details of the ideal-gas assumption are summarized in \ref{app::ideal-gas}.
The root mean square (RMS) error is used to measure the truncation error in a given physical quantity, which is defined as
\begin{equation}\label{RMS}
	\mathscr{E}_{\text{RMS}} := \sqrt{\frac{1}{N}\sum_{i=1}^{N}\left|\mathscr{E}_i\right|^2},
\end{equation}
where $\mathscr{E}_i$ is the absolute error in the physical quantity due to $i$th particle, and $N$ is the total number of particles.

In the following analysis, we denote the cutoff radii in real and Fourier spaces as $r_c$ and $k_c$, i.e., one only calculate the terms satisfying $\abs{\bm{r}_{ij} + \mathcal{M}} \leq r_c$ and 
$|\bm{k}|\leq k_c$ in real and Fourier spaces, respectively.
Our main findings are summarized as follows. 

\begin{thm}\label{thm:ewald2d_phi_error}
	Under the ideal-gas assumption, the real space and Fourier space truncation errors for the Ewald2D summation can be estimated by
	\begin{equation}\label{eq::Ephi}
		\mathscr{E}_{\phi_s} (r_c, \alpha) \approx \sqrt{\frac{4\pi Q}{V}\mathscr{Q}_{s}(\alpha,r_c)}\;, \quad
		\mathscr{E}_{\phi_\ell} (k_c, \alpha) \approx \sqrt{\frac{8\alpha^2Q}{\pi V}}k_c^{-3/2}e^{-k_c^2/(4\alpha^2)}\;,
	\end{equation}
	where $Q = \sum_{i=1}^{N} q_{i}^2$ and
	\begin{equation}\label{eq::Qs2}
		\mathscr{Q}_{s}(\alpha,r_c):=\frac{2e^{-\alpha^2r_c^2}\erfc(\alpha r_c)}{\alpha\sqrt{\pi}}-r_c\erfc(\alpha r_c)^2-\sqrt{\frac{2}{\pi\alpha^2}}\erfc(\sqrt{2}\alpha r_c).
	\end{equation}
	Notably,
	\begin{equation}\label{eq::Qs}
		\mathscr{Q}_{s}(\alpha,r_c)\rightarrow\frac{1}{4\pi}\alpha^{-4} r_c^{-3}e^{-2\alpha^2r_c^2}~\text{as}~\alpha r_c\rightarrow \infty\;. 
	\end{equation}
\end{thm}
The proof of Theorem~\ref{thm:ewald2d_phi_error} is provided in~\ref{app:phierr}.
An interesting observation is that at the limit $\alpha r_c\rightarrow \infty$, the truncation error estimates for Ewald2D sum become identical as that for Ewald3D derived in~\cite{kolafa1992cutoff}.
Same observation has been made by Tornberg and her coworkers through numerical tests~\cite{lindbo2012fast,shamshirgar2021fast}. 
Here, Theorem~\ref{thm:ewald2d_phi_error} justifies this phenomenon.

Based on Theorem~\ref{thm:ewald2d_phi_error}, one can further obtain the error estimates of the interaction energy and forces, summarized in Proposition \ref{prop::2.12}.

\begin{prop}\label{prop::2.12}
	Under the ideal-gas assumption, the real space and Fourier space RMS errors of energy and forces by the truncated Ewald2D summation can be estimated by
	\begin{equation}\label{thm:ewald2d_error}
		\mathscr{E}_{U_s} (r_c, \alpha) \approx Q \sqrt{\frac{1}{2 V}} \alpha^{-2} r_c^{-3/2} e^{-\alpha^2r_c^2}\;,\quad
		\mathscr{E}_{U_{\ell}} (k_c, \alpha) \approx Q \sqrt{\frac{8 \alpha^2}{\pi V}} k_c^{-3/2} e^{- k_c^2/(4 \alpha^2)}\;,
	\end{equation}
	and
	\begin{equation}
		\mathscr{E}_{\bm{F}_{s}^i} (r_c, \alpha)\approx 2|q_{i}|\sqrt{\frac{Q}{V}}r_c^{-1/2}e^{-\alpha^2 r_c^2},\quad \mathscr{E}_{\bm{F}_{\ell}^i} (k_c, \alpha)\approx 4|q_{i}|\sqrt{\frac{Q}{\pi V}}\alpha k_c^{-1/2}e^{-k_c^2/(4\alpha^2)}\;,
	\end{equation}
	as $\alpha r_c\rightarrow\infty$ and $k_c/2\alpha\rightarrow\infty$, respectively.
\end{prop}

\begin{rmk}
	In practice, one needs to pick the pair of~$r_c$ and~$k_c$
	such that the series in real and Fourier spaces converge with the same speed.
	By Theorem~\ref{thm:ewald2d_phi_error}, they can be chosen as
	\begin{equation}\label{eq::rckc}
		r_c = \frac{s}{\alpha},~\hbox{and}~k_c = 2 s \alpha\;,
	\end{equation}
	such that both truncation errors decay as 
	\begin{equation}\label{eq:trunction_error}
		\mathscr{E}_{\phi_s}(r_c, \alpha)\approx \mathscr{E}_{\phi_\ell}(k_c, \alpha) \sim Q 
		\sqrt{\frac{s}{\alpha V}} \frac{e^{-s^2}}{s^2}\;.
	\end{equation}
	This indicates that the trunction error can be well controlled by the prescribed parameter $s$. 
\end{rmk}

It should be noticed that, since the real space interaction is short-ranged, it only requires computation of neighboring pairs within the cutoff radius $r_c$. 
Many powerful techniques have been developed to reduce the cost for such short-range interactions into $\mathcal O(N)$ complexity, including the Verlet list~\cite{verlet1967computer}, the linked cell list~\cite{allen2017computer} and more recently the random batch list~\cite{liang2021random} algorithms. 
Consequently, the main challenge lies in the long-range component calculation, which will be discussed in Section~\ref{sec:method}.

\subsection{Extension to systems with charged slabs}
\label{sec::sysslabs}
In the presence of charged slabs, boundary layers naturally arise -- opposite ions accumulate near the interface, forming an electric double layer. The structure of electric double layers plays essential role for properties of interfaces and has caught much attention~\cite{messina2004effect,breitsprecher2014coarse,moreira2002simulations}. Since charges on the slabs are often represented as a continuous surface charge density, we present the Ewald2D formulation with such a situation can be well treated.

Without loss of generality, one assumes that the two charged slab walls are located at $z=0$ and $z=L_z$ and with smooth surface charge densities $\sigma_{\mathrm{bot}}(\bm{\rho})$ and $\sigma_{\mathrm{top}}(\bm{\rho})$, respectively. Note that both $\sigma_{\mathrm{bot}}(\bm{\rho})$ and $\sigma_{\mathrm{top}}(\bm{\rho})$ are doubly-periodic according to the quasi-2D geometry. In such cases, the charge neutrality condition of the system reads
\begin{equation}\label{eq::chargeneu}
 	\sum_{i=1}^{N}q_{i} + \int_{\mathcal{R}^2} \left[\sigma_{\mathrm{top}} (\bm{\rho}) + \sigma_{\mathrm{bot}}(\bm{\rho}) \right] d\bm{\rho} = 0.
\end{equation} 
Under such setups, the potential $\phi$ can be written as the sum of particle-particle and particle-slab contributions,
\begin{align}
	\phi(\bm{r})=\phi_{\text{p-p}}(\bm{r})+\phi_{\text{p-s}}(\bm{r}).
\end{align}
Here, $\phi_{\text{p-p}}$ satisfies Eq.~\eqref{eq::Poisson} associated with the boundary condition Eq.~\eqref{eq::boundary2}. Note that Eq.~\eqref{eq::boundary1} does not apply since the particles are overall non-neutral. $\phi_{\text{p-s}}$ satisfies 
\begin{equation}\label{eq::PoionWall}
	-\Delta \phi_{\text{p-s}}(\bm{r}) = 4\pi h(\bm{r}), \quad \text{with}~ h(\bm{r}) =  \sigma_{\mathrm{bot}}(\bm{\rho}) \delta(z) + \sigma_{\mathrm{top}}(\bm{\rho}) \delta(z-L_z),
\end{equation}
with the boundary condition
\begin{equation}\label{eq::boundionwall}
	\lim_{z\rightarrow\pm\infty}\phi_{\text{p-s}}(\bm{r})=\mp \frac{2\pi}{L_xL_y}\left(\int_{\mathcal{R}^2}\sigma_{\mathrm{bot}}(\bm{\rho})|z|d\bm{\rho}+\int_{\mathcal{R}^2}\sigma_{\mathrm{top}}(\bm{\rho})|z-L_z|d\bm{\rho}\right)
\end{equation}
which is simply the continuous analog of Eq.~\eqref{eq::boundary2}. 

The potential $\phi_{\text{p-p}}$ then follows immediately from Lemma~\ref{thm::SpectralExpansion}
\begin{equation}\label{eq::phiion-ion}
	\phi_{\text{p-p}}(\bm{r}_{i})=\phi_{s}(\bm{r}_{i}) + \sum_{\bm{k}\neq\bm{0}} \phi_{\ell}^{\V{k}}(\bm{r}_{i}) + \phi_{\ell}^{\V{0}}(\bm{r}_{i}) - \phi_{\text{self}}^{i},
\end{equation}
with each components given by Eqs.~\eqref{eq:phi_s}, \eqref{eq:philk}, \eqref{eq:phil0}, and \eqref{eq::self}, respectively. 
The 2D Fourier series expansion of~$\phi_{\text{p-s}}$ is provided in the following Theorem~\ref{thm::ionwall}, where its convergence rate is controlled by the smoothness of surface charge densities.
\begin{thm}\label{thm::ionwall}
	Suppose that $\widehat{\sigma}_{\mathrm{bot}}$ and $\widehat{\sigma}_{\mathrm{top}}$ are two-dimensional Fourier transform (see Lemma~\ref{lem::2dfourier}) of $\sigma_{\mathrm{bot}}$ and $\sigma_{\mathrm{top}}$, respectively. By Fourier analysis, the particle-slab component of the electric potential is given by
	\begin{equation}\label{eq::phiionwall}
		\phi_{\emph{p-s}}(\bm{r}_{i}) = \frac{2\pi}{L_xL_y}\sum_{\bm{k}\neq \bm{0}}\frac{e^{\m{i} \bm{k}\cdot\bm{\rho}_{i}}}{k}\left[\widehat{\sigma}_{\mathrm{bot}}(\bm{k})e^{-k|z_{i}|}+\widehat{\sigma}_{\mathrm{top}}(\bm{k})e^{-k|z_{i}-L_z|}\right]+\phi_{\emph{p-s}}^{\bm{0}}(\bm{r}_{i})\;,
	\end{equation}
	where the 0-th mode reads
	\begin{equation}\label{eq::phionwallzero}
		\phi_{\emph{p-s}}^{\bm{0}}(\bm{r}_{i})=-\frac{2\pi}{L_xL_y}\Big[\widehat{\sigma}_{\mathrm{bot}}(\bm{0})|z_{i}|+\widehat{\sigma}_{\mathrm{top}}(\bm{0})|z_{i}-L_z|\Big]\;.
	\end{equation}
\end{thm}
\begin{proof}
	For $\bm{k}\neq\bm{0}$, applying the quasi-2D Fourier transform to both sides of Eq.~\eqref{eq::PoionWall} yields
	\begin{equation}
		\widetilde{\phi}_{\text{p-s}}(\bm{k},\kappa)=\frac{4\pi}{k^2+\kappa^2}\left[\widehat{\sigma}_{\mathrm{bot}}(\bm{k})+\widehat{\sigma}_{\mathrm{top}}(\bm{k})e^{-\m{i} \kappa L_z}\right]\;.
	\end{equation}
	For $\bm{k}=\bm{0}$, one first applys the 2D Fourier transform in $xy$ to obtain
	\begin{equation}
		\left(-\partial_z^2+k^2\right)\widehat{\phi}(\bm{k},z)=4\pi\left[\widehat{\sigma}_{\mathrm{bot}}(\bm{k})\delta(z)+\widehat{\sigma}_{\mathrm{top}}(\bm{k})\delta(z-L_z)\right]\;.
	\end{equation}
	By integrating both sides twice and taking $\bm{k}=\bm{0}$, the $0$-th mode follows 
	\begin{equation}
		\widehat{\phi}(\bm{0},z)=-2\pi\left[\widehat{\sigma}_{\mathrm{bot}}(\bm{0})|z|+\widehat{\sigma}_{\mathrm{top}}(\bm{0})|z-L_z|\right]+A_0z+B_0\;,
	\end{equation}
	where $A_0$ and $B_0$ are undetermined constants. 
	Finally, applying the corresponding inverse transforms to $\widetilde{\phi}_{\text{p-s}}(\bm{k},\kappa)$ and $\widehat{\phi}(\bm{0},z)$ such that the boundary conditions Eq.~\eqref{eq::boundionwall} is matched, one has $A_0=B_0=0$. The proof of Eqs.~\eqref{eq::phiionwall}-\eqref{eq::phionwallzero} is then completed. 
\end{proof}

Consider the ideal case that both $\sigma_{\mathrm{bot}}$ and $\sigma_{\mathrm{top}}$ are uniformly distributed. This simple setup is widely used in many studies on interface properties. Since in this case all nonzero modes vanish, one has
\begin{equation}\label{eq:spectial}
	\phi_{\text{p-s}}(\bm{r}_{i})=\phi_{\text{p-s}}^{\bm{0}}(\bm{r}_{i})=-2\pi\left[\sigma_{\mathrm{top}}(L_z - z_{i})+\sigma_{\mathrm{bot}}(z_{i} - 0))\right]\;,
\end{equation}
for all $z_{i}\in [0, L_z]$. 
Here zero is retained to indicate the location of bottom slab. 

For completeness, Proposition \ref{welldefinedness} provides the result of the well-definedness.
\begin{prop} \label{welldefinedness}
	The total electrostatic potential $\phi$ is well-defined. 
\end{prop}
\begin{proof}
	For any finite $z$, $\phi$ is clearly well defined. Consider the case of $z\rightarrow \pm \infty$.
	By boundary conditions ~\eqref{eq::boundary2} and \eqref{eq::boundionwall} and the charge neutrality condition Eq.~\eqref{eq::chargeneu}, one has
	\begin{equation}
		\begin{split}
			\lim_{z\rightarrow \pm \infty}\phi(\bm{r})&=\lim_{z\rightarrow \pm \infty}\left[\phi_{\text{p-p}}(\bm{r})+\phi_{\text{p-s}}(\bm{r})\right]\\
			&=\pm\frac{2\pi}{L_xL_y}\left[\sum_{j=1}^{N}q_{j}z_{j}+\int_{\mathcal{R}^2}\left(0\sigma_{\mathrm{bot}}(\bm{\rho})+L_z\sigma_{\mathrm{top}}(\bm{\rho})\right)d\bm{\rho}\right]
		\end{split}
	\end{equation}
	which is a finite constant. Thus the proof is completed.
\end{proof}

For the the particle-slab interaction formulation, we observe a constant discrepancy between Eq.~\eqref{eq:spectial} derived here and those in literature~\cite{dos2017simulations,10.1063/1.4998320}. 
It is because here one starts with the precise Ewald2D summation approach, different from the approach of employing approximation techniques to transform the original doubly-periodic problem into a triply-periodic problem first, and subsequently introducing charged surfaces. 
This constant discrepancy makes no difference in force calculations for canonical ensembles. 
However, for simulations under isothermal-isobaric ensembles, this $L_z$-dependent value is important for the pressure calculations~\cite{li2024noteaccuratepressurecalculations}. 
And one should use Eq.~\eqref{eq:spectial} derived here for correct simulations.

Based on the expression of electrostatic potential $\phi$ derived above, the total electrostatic energy can be computed via the Ewald2D summation formula:  
\begin{align}\label{eq::34}
	U = U_{\text{p-p}} + U_{\text{p-s}}, \quad \text{with} \quad U_{\text{p-p}} := U_{s} + \sum_{\bm{k}\neq\bm{0}}U_{\ell}^{\bm{k}}+U_{\ell}^{\bm{0}}- U_{\text{self}}\;,
\end{align}
where $U_*=\sum_{i}\phi_*$ with $*$ representing any of the subscripts used in Eq.~\eqref{eq::34}.

\section{Sum-of-exponentials Ewald2D method} \label{sec:method}

In this section, we introduce a novel summation method by using the SOE approximation in evaluating $\xi^{\pm}(k,z)$ and $\partial_{z}\xi^{\pm}(k,z)$. 
This method significantly reduces the overall complexity of Ewald2D to $O(N^{7/5})$ without compromising accuracy. 
Error and complexity analyses are also provided.

We first give a brief overview of the SOE kernel approximation method.  
For a given precision $\varepsilon$, the objective of an SOE approximation is to find suitable weights $w_l$ and exponents $s_l$ such that~$\forall x \in \mathbb{R}$, the following inequality holds:
\begin{equation}\label{eq::SOE1}
	\left|f(x)-\sum_{l=1}^M w_l e^{-s_l|x|}\right|\leq \varepsilon,
\end{equation}
where $M$ is the number of exponentials. 
Various efforts have been made in literature to approximate different kernel functions using SOE, as documented in works such as~\cite{wiscombe1977exponential, evans1980least, jiang2008efficient, beylkin2005approximation, beylkin2010approximation}.
For instance, the Gaussian kernel $f(x)=e^{-x^2}$ is widely celebrated and plays a crucial role in numerical PDEs~\cite{weideman2010improved,wang2018adaptive}, and is particularly relevant for the purpose of this work. 
The SOE approximation for Gaussians can be understood as discretizing its inverse Laplace transform representation, denoted as
\begin{equation}\label{eq:inverse_laplace}
	e^{-x^2} = \frac{1}{2 \pi \m{i}} \int_{\Gamma} e^{z} \sqrt{\frac{\pi}{z}} e^{-\sqrt{z}\abs{x}} dz\;,
\end{equation}
where $\Gamma$ is a suitably chosen contour. 

To achieve higher accuracy, several classes of contours have been studied, such as Talbot contours~\cite{lin2004numerical}, parabolic contours~\cite{makarov2000exponentially}, and hyperbolic contours~\cite{lopez2004numerical}. 
An alternative approach is developed by Trefethen \emph{et al.}~\cite{trefethen2006talbot}, where a sum-of-poles expansion is constructed by the best supremum-norm rational approximants. 
A comprehensive review of these techniques has been discussed by Jiang and Greengard~\cite{jiang2021approximating}.
Since the Laplace transform of an SOE is a sum-of-poles expansion~\cite{greengard2018anisotropic}, the model reduction (MR) technique can be employed to further reduce the number of exponentials $M$ while achieving a specified accuracy~$\eps$.
When combining with the MR, convergence rates at $\mathcal{O}(6^{-M}) \sim \mathcal{O}(7^{-M})$ can be achieved~\cite{jiang2021approximating}.

Additionally, kernel-independent SOE methods have been developed, such as the black-box method~\cite{greengard2018anisotropic} and Vall\'ee-Poussin model reduction (VPMR) method~\cite{gao2021kernelindependent}. 
Specially, the VPMR method integrates the flexibility of Vall\'ee-Poussin sums into the MR technique, demonstrating the highest convergence rate of $\mathcal{O}(9^{-M})$ in constructing SOE approximation for Gaussians. 
This method is also bandwidth-controllable and uniformly convergent~\cite{AAMM-13-1126}. 
Due to these advantages, we will utilize the VPMR as the SOE construction tool in all the numerical experiments throughout this paper.



\subsection{SOE approximations of $\xi^{\pm}(k,z)$} \label{subsec::SOEapp}

To start with, we introduce a useful identity which is a special case of the Laplace transform~(\cite{oberhettinger2012tables}, pp.~374-375; \cite{myint2007linear}, pp.~688). 
\begin{lem}
	Suppose that $a$, $b$, and $c$ are three complex parameters where the real part of $a$ satisfies $\mathscr{R}(a)>0$. For an arbitrary real variable $x$, the following identity holds:
	\begin{equation}\label{eq::identity}
		\int_{x}^{\infty} e^{-(at^2+2bt+c)}dt=\frac{1}{2}\sqrt{\frac{\pi}{a}}e^{(b^2-ac)/a}\erfc\left(\sqrt{a}x+\frac{b}{\sqrt{a}}\right).
	\end{equation}
\end{lem}

Substituting $a=\alpha^2$, $b=k/2$, $c=0$, $x=\pm z$ into Eq.~\eqref{eq::identity} yields the integral representations of $\xi^{\pm}(k,z)$:
\begin{equation}\label{eq::xi20}
	\xi^{\pm}(k,z) = \frac{2\alpha}{\sqrt{\pi}}e^{-k^2/(4\alpha^2)}e^{\pm kz}\int_{\pm z}^{\infty}e^{-\alpha^2t^2-kt} dt \;.
\end{equation}
We then approximate the Gaussian factor $e^{-\alpha^2t^2}$ in the integrand of Eq.~\eqref{eq::xi20} by an $M$-term SOE on the whole real axis, as
\begin{equation}\label{eq::xi21}
	e^{-\alpha^2 t^2} \approx \sum_{l = 1}^M w_{l} e^{ - s_{l} \alpha \abs{t}}\;.
\end{equation}
Inserting Eq.~\eqref{eq::xi21} into Eq.~\eqref{eq::xi20} results in an approximation to $\xi^{\pm}(k,z)$:
\begin{equation}\label{eq::xi_M_int}
	\xi_{M}^{\pm}(k, z) := \frac{2\alpha}{\sqrt{\pi}} e^{-k^2/(4\alpha^2)} e^{\pm kz} \int_{\pm z}^{\infty} \sum_{l = 1}^M w_{l} e^{ - s_{l} \alpha \abs{t}} e^{-kt} dt \;.
\end{equation}
The integral can be calculated analytically (with $\alpha, z > 0$), yielding 
\begin{equation}\label{eq::plus}
	\xi_M^{+}(k,z) = \frac{2 \alpha}{\sqrt{\pi}} e^{-k^2/(4 \alpha^2)} \sum_{l = 1}^M w_l\frac{e^{- \alpha s_l z}}{\alpha s_l + k } 
\end{equation}
and
\begin{equation}\label{eq::minus}
	\xi_{M}^{-}(k, z) = \frac{2 \alpha}{\sqrt{\pi}} e^{-k^2/(4 \alpha^2)} \sum_{l = 1}^M w_l \left[ - \frac{e^{-\alpha s_l z}}{\alpha s_l - k} + \frac{2\alpha s_l e^{- kz}} {(\alpha s_l)^2 - k^2} \right] \;.
\end{equation}
Similarly, one can also obtain the approximation of~$\partial_z \xi^{\pm}(k,z)$, given by
\begin{equation}\label{eq::dz_plus}
	\partial_z \xi_{M}^{+}(k, z) := - \frac{2 \alpha^2 }{\sqrt{\pi}} e^{-k^2/(4 \alpha^2)} \sum_{l = 1}^M w_l s_l \frac{e^{- \alpha s_l z}}{\alpha s_l + k } 
\end{equation}
and
\begin{equation}\label{eq::dz_minus}
	\partial_z \xi_{M}^{-}(k, z) := - \frac{2 \alpha^2}{\sqrt{\pi}} e^{-k^2/(4 \alpha^2)} \sum_{l = 1}^M w_l s_l \left[ - \frac{e^{-\alpha s_lz}}{\alpha s_l - k} + \frac{2 k e^{- kz}} {(\alpha s_l )^2 - k^2} \right] \;.
\end{equation}
The approximation error, which relies on the prescribed precision $\varepsilon$ of the SOE, also has spectral convergence in $k$.
This is summarized in Theorem \ref{thm:error_xi}.

\begin{thm}\label{thm:error_xi}
	Given an $M$-term SOE expansion for the Gaussian kernel $f(x)=e^{-\alpha^2x^2}$ satisfying Eq.~\eqref{eq::SOE1}, the approximation of~$\xi^{\pm}$ derived from Eqs.~\eqref{eq::xi20}-\eqref{eq::minus} has a global error bound
	\begin{equation}\label{eq::bound_xi}
		\abs{\xi^{\pm}(k,z) - \xi_{M}^{\pm}(k, z)} \leq \frac{2 \alpha e^{-k^2/(4 \alpha^2)}}{\sqrt{\pi} k} \varepsilon\;,
	\end{equation}
	and for the approximation of~$\partial_z \xi^{\pm}$ using Eqs.~\eqref{eq::dz_plus}-\eqref{eq::dz_minus}, the error bound is given by
	\begin{equation}\label{eq::bound_dz_xi}
		\abs{\partial_z \xi^{\pm}(k,z) - \partial_z \xi_{M}^{\pm}(k, z)} \leq \frac{4 \alpha e^{-k^2/(4 \alpha^2)}}{\sqrt{\pi}} \varepsilon\;,
	\end{equation}
	which is independent of $z$ and decays rapidly with $k$.
\end{thm}

\begin{proof}
	To prove Eq.~\eqref{eq::bound_xi}, one can directly subtract Eq.~\eqref{eq::xi_M_int} from  Eq.~\eqref{eq::xi20} to obtain:
	\begin{equation}
		\begin{split}
			&\abs{\xi^{\pm}(k, z) - \xi_{M}^{\pm}(k, z)}\\
			\leq & \frac{2 \alpha}{\sqrt{\pi}} e^{-k^2/(4 \alpha^2)} \int_{\pm z}^\infty e^{\pm kz - k t} \abs{e^{-\alpha^2 t^2} - \sum_{l = 1}^M w_l e^{-\alpha s_l\abs{t}} } dt \\
			\leq & \frac{2 \alpha}{\sqrt{\pi}} e^{ - k^2/(4 \alpha^2)} \varepsilon \int_{\pm z}^\infty e^{\pm kz - k t} dt \\
			= & \frac{2 \alpha e^{-k^2/(4 \alpha^2)}}{\sqrt{\pi} k} \varepsilon\;,
		\end{split}
	\end{equation}
	where from the second to the third line, one uses the boundness of SOE approximation error, i.e., Eq.~\eqref{eq::SOE1}.
	For the approximation error of $\partial_z\xi^{\pm}$, the proof is similar. One can subtract the $z$-derivative of Eq.~\eqref{eq::xi_M_int} from the $z$-derivative of Eq.~\eqref{eq::xi20} to obtain
	\begin{equation}
	\begin{split}
	& \abs{\partial_z \xi^{\pm}(k, z) - \partial_z \xi_{M}^{\pm}(k, z)}\\
	\leq & \frac{2 \alpha}{\sqrt{\pi}} e^{-k^2/(4 \alpha^2)} \left( \int_{\pm z}^\infty k e^{\pm kz - k t} \abs{e^{-\alpha^2 t^2} - \sum_{l = 1}^M w_l e^{-\alpha s_l\abs{t}} } dt + \abs{e^{-\alpha^2 z^2} - \sum_{l = 1}^M w_l e^{-\alpha s_l\abs{z}} } \right)\\
	\leq & \frac{2 \alpha}{\sqrt{\pi}} e^{ - k^2/(4 \alpha^2)} \varepsilon \left( k \int_{\pm z}^\infty e^{\pm kz - k t} dt + 1 \right) \\
	= & \frac{4 \alpha e^{-k^2/(4 \alpha^2)}}{\sqrt{\pi}} \varepsilon\;.
	\end{split}
	\end{equation}
	Again, the boundness of the SOE approximation error is used in the proof, i.e., from the second to the third line.
\end{proof}

For the $\V 0$-th frequency term Eq.~\eqref{eq:phil0} consisting of $\erf(\cdot)$ and Gaussian functions, a similar approach can be employed to construct the corresponding SOE expansions. One has
\begin{equation}\label{eq:SOErf}
	\begin{split}
		\erf(\alpha z)\approx \frac{2}{\sqrt{\pi}} \int_0^{\alpha z} \sum_{l = 1}^M w_l e^{-s_l t} dt= \frac{2}{\sqrt{\pi}} \sum_{l = 1}^M \frac{w_l}{s_l} (1 - e^{- \alpha s_l z})\;.
	\end{split}
\end{equation}
One can prove that
\begin{equation}\label{eq::erfSOE}
	\left|\erf(\alpha z)-\frac{2}{\sqrt{\pi}} \sum_{l = 1}^M \frac{w_l}{s_l} (1 - e^{- \alpha s_l z})\right|\leq \frac{2\alpha L_z}{\sqrt{\pi}}\varepsilon\;,
\end{equation}
where one assumes that $z\in[0,L_z]$, as for quasi-2D systems all particles are confined within a narrow region in $z$. 

The advantages and novelty of our SOE approach are summarized as follows.
Firstly, Theorem~\ref{thm:error_xi} demonstrates that the approximation error of our SOE method is uniformly controlled in $z$ and decays exponentially in $k$. Secondly, the resulting approximation $\xi_{M}^{\pm}$ is well-conditioned, thus addressing the issue of catastrophic error cancellation when evaluating $\xi^{\pm}$. 
Achieving these properties are crucial for the subsequent algorithm design.

\begin{rmk}
	The choice of SOE approximation is not unique. Instead of approximating the Gaussian factor in the integrand with an SOE, a more straightforward approach might be approximating the complementary error function in~$\xi^{\pm}$ using an SOE. 
	Unfortunately, this will introduce an error proportional to~$e^{kz} \varepsilon$, which grows exponentially with~$k$, resulting in a much larger numerical error in the Fourier space summation for the long-range components of Coulomb energy and forces.
\end{rmk}

\begin{figure}[ht] 
	\centering
	\includegraphics[width=\textwidth]{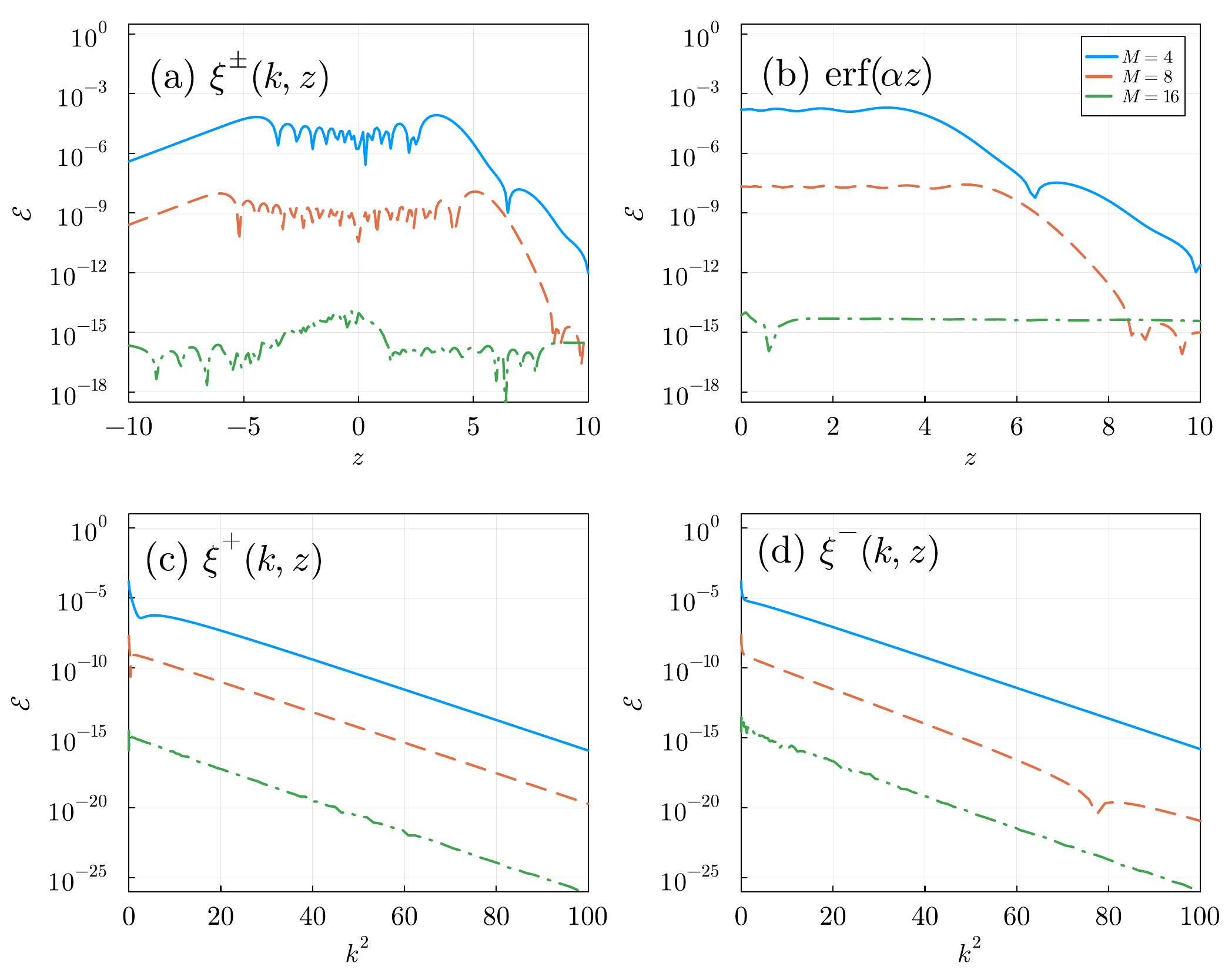}
	\caption{
		The absolute error of the SOE expansion for (a) $\xi^{\pm}(k,z)$ and (b) $\erf(\alpha z)$ is plotted as a function of $z$, while fixing $k=\alpha=1$; absolute error of the SOE expansion of (c) $\xi^{+}(k,z)$ and (d) $\xi^{-}(k,z)$ as a function of $k^2$, while fixing $z=1$.
		Data are presented for SOEs with varying numbers of exponentials, $M=4,$ 8 and 16.
	}
	\label{fig:error_SOE}
\end{figure}


\subsection{SOEwald2D summation and its fast evaluation}\label{subsec::reforEwald2}
In this section, we derive the SOE-reformulated Ewald2D (SOEwald2D) summation, and the corresponding fast evaluation scheme. 
Let us first consider the contribution of the $\V{k}$-th mode $(\V k \neq \V{0})$ to the long-range interaction energy, denoted as $U_{\ell}^{\V{k}}$, which can be written in the following pairwise summation form 
\begin{equation}\label{eq::pairssum8}
	\begin{split}
		U_{\ell}^{\V{k}} = \frac{1}{2} \sum_{i = 1}^N q_{i} \phi^{\V{k}}_\ell(\V{r}_{i}) =\frac{\pi}{L_xL_y}\sum_{1\leq j<i\leq N}q_{i}q_{j} \varphi^{\bm{k}} (\bm{r}_{i},\bm{r}_{j}) + \frac{Q \pi}{k L_x L_y} \erfc\left(\frac{k}{2 \alpha}\right)\;,
	\end{split}
\end{equation}
where we define
\begin{equation}\label{eq:varphi_k}
	\varphi^{\bm{k}}(\bm{r}_{i},\bm{r}_{j}):=\frac{e^{\m{i} \V{k} \cdot \V{\rho}_{ij}}}{k} \left[\xi^{+}(k, z_{ij})+\xi^{-}(k, z_{ij})\right].
\end{equation}
Substituting the SOE approximation of $\xi^{\pm}(k,z)$ described in Eqs.~\eqref{eq::plus} and \eqref{eq::minus}, a new SOE-based formulation can be obtained, denoted as $U_{\ell,\text{SOE}}^{\V{k}}$. 
This approximation is achieved by substituting $\varphi^{\bm{k}}(\bm{r}_{i},\bm{r}_{j})$ with
\begin{equation}\label{eq::SOEphi}
	\begin{split}
		\varphi^{\bm{k}}_{\text{SOE}}(\bm{r}_{i},\bm{r}_{j}) & := \frac{e^{\m{i} \V{k} \cdot \V{\rho}_{ij}}}{k} \left[\xi^{+}_M(k, z_{ij})+\xi^{-}_M(k, z_{ij})\right]\\
		&=\frac{2\alpha e^{-k^2/(4\alpha^2)}}{\sqrt{\pi}k}e^{\m{i}\V{k} \cdot \V{\rho}_{ij}} \sum_{\ell = 1}^{M}  \frac{w_l}{\alpha^2 s_l^2 - k^2}\left( 2 \alpha s_l e^{-k z_{ij}} - 2ke^{-\alpha s_l z_{ij}}\right) \;.
	\end{split}
\end{equation}

For the $\V 0$-th mode contribution $U_{\ell}^{\V{0}}$, an SOE-based reformulation can be similarly obtained according to Eq.~\eqref{eq:SOErf}:
\begin{equation}\label{eq::Ul0SOE}
	U_{\ell,\text{SOE}}^{\bm{0}} = \frac{1}{2} \sum_{i = 1}^N q_{i} \phi^{\bm{0}}_\ell(\V{r}_{i}) = - \frac{2\pi}{L_xL_y}\sum_{1\leq j < i \leq N}q_{i}q_{j}\varphi^{\bm{0}}_{\text{SOE}}(\bm{r}_{i},\bm{r}_{j})-\frac{\pi Q}{\alpha L_xL_y},
\end{equation}
where
\begin{equation}\label{eq::32}
	\varphi^{\bm{0}}_{\text{SOE}}(\bm{r}_{i},\bm{r}_{j}):=\sum_{l=1}^{M}\frac{w_l}{\sqrt{\pi}}\left[\frac{2z_{ij}}{s_l}+\left(\frac{1}{\alpha}-\frac{2z_{ij}}{s_l}\right)e^{-\alpha s_l z_{ij}}\right].
\end{equation}

We now present an iterative approach to compute $U_{\ell,\text{SOE}}^{\V{k}}$ for each $\bm{k}$ with $\mathcal O(N)$ complexity. 
For simplicity, consider pairwise sum in the following form:
\begin{equation}\label{eq::33}
	S = \sum_{1\leq j < i \leq N} q_{i} q_{j} e^{\m{i} \V{k} \cdot \V{\rho}_{ij}} e^{ - \beta z_{ij} }\;,
\end{equation}
where $\beta$ is a parameter satisfying $\mathscr{Re}(\beta)>0$. It is clear that the pairwise sums in both $U_{\ell,\text{SOE}}^{\V{k}}$ and $U_{\ell,\text{SOE}}^{\V{0}}$ are in the form of Eq.~\eqref{eq::33}, and direct evaluation takes $\mathcal O(N^2)$ cost. 

To efficiently evaluate $S$, we initially sort the particle indices based on their $z$ coordinates, such that~$i > j \iff z_i > z_j$.
Subsequently, the summation can be rearranged into a separable and numerically stable form:
\begin{align}
	S &=\sum_{i = 1}^N q_{i} e^{\m{i} \bm{k}\cdot\bm{\rho}_{i}}e^{-\beta z_{i}} \sum_{j = 1}^{i - 1} q_{j} e^{-\m{i} \bm{k}\cdot\bm{\rho}_{j}}e^{\beta z_{j}} \label{eq::S0} \\ 
	&= \sum_{i = 1}^N q_{i} e^{\m{i} \bm{k}\cdot\bm{\rho}_{i}}e^{-\beta(z_{i} - z_{i - 1})}A_{i}(\beta)\;, \label{eq::S}
\end{align}
with coefficients
\begin{equation}
	A_{i}(\beta) =  \sum_{j = 1}^{i - 1} q_{j} e^{-\m{i} \bm{k}\cdot\bm{\rho}_{j}}e^{-\beta(z_{i - 1} - z_{j})}\;.
\end{equation}
Clearly, $A_1(\beta)=0,$ $A_2(\beta)=q_1e^{-\m{i} \bm{k}\cdot\bm{\rho}_1}$, and for $i\geq 3$, a recursive algorithm can be constructed to achieve $\mathcal O(N)$ complexity in computing all the coeffcients:
\begin{equation}\label{eq::35}
	A_{i}(\beta) = A_{i - 1}(\beta) e^{-\beta(z_{i - 1} - z_{i - 2})} + q_{i - 1}e^{-\m{i} \bm{k}\cdot\bm{\rho}_{i-1}}\;,\quad i = 3, \cdots, N.
\end{equation}
One can thus efficiently evaluate $S$ with another $\mathcal{O}(N)$ operations by Eq.~\eqref{eq::S} and using the computed coefficients~$\V{A}(\beta) = (A_1(\beta), ..., A_N(\beta))$. 
Consequently, the overall cost for evaluating the pairwise summation in forms of Eq.~\eqref{eq::33} is reduced to $\mathcal{O}(N)$.
Besides the iterative method discussed above, it is remarked that different fast algorithms based on SOE for 1D kernel summations have been developed~\cite{jiang2021approximating,GIMBUTAS2020815}, which can also be used under the framework described in this article for the summation in the non-periodic direction.

\begin{rmk}
	A similar iterative evaluation strategy can be developed based on Eq.~\eqref{eq::S0} instead of Eq.~\eqref{eq::S}, which may seem more straightforward.
	However, it will lead to uncontrolled exponential terms such as $e^{\beta z_{j}}$, affecting the numerical stability.
	The same issue occurs in the original Ewald2D summation, as has been discussed in Section~\ref{subsec::elec}.
In our recursive scheme, by prior sorting of all particles in $z$, it follows that  $z_{i}-z_{i-1} > 0$ and $z_{i-1} - z_{j} \geq 0$ for all $j \leq i - 1$, thus making all exponential terms in Eq.~\eqref{eq::S} with negative exponents, resolving the exponential blowup issue.
	\label{rmk::stable}
\end{rmk}

Finally, the long-range component of Coulomb interaction energy in Ewald2D summation is approximated via:
\begin{equation}\label{eq:U_l_SOE}
	U_{\ell}\approx U_{\ell,\text{SOE}} := \sum_{\bm{k}\neq \bm{0}} U_{\ell,\text{SOE}}^{\bm{k}}+U_{\ell,\text{SOE}}^{\bm{0}},
\end{equation}
where both $U_{\ell,\text{SOE}}^{\bm{k}}$ and $U_{\ell,\text{SOE}}^{\bm{0}}$ can be evaluated efficiently and accurately with linear complexity.
In MD simulations, the force exerts on the $i$-th particle, $\bm{F}_{i}$, plays a significant role in the numerical integration of Newton's equations.
One can similarly develop fast recursive algorithms to evaluate the SOE-reformulated forces,
the detailed expressions for $\bm{F}_{i}$ is summarized in \ref{app::force}.
We finally summarize the SOEwald2D  in Algorithm~\ref{alg:SOEwald2D}. Its error and complexity analysis will be discussed in the next sections.

\begin{algorithm}[ht] 
	\caption{The sum-of-exponentials Ewald2D method}
	\begin{algorithmic}[1]
		\setstretch{1.15}
		\State \textbf{Input}: Initialize the size of the simulation box $(L_x, L_y, L_z)$, as well as the positions, velocities, and charges of all particles. Choose a precision requirement $\varepsilon$.  
		
		\State \textbf{Precomputation stage}: Determine Ewald splitting parameters $\alpha$ and $s$ according to Eq.~\eqref{eq:trunction_error}. Generate real and Fourier space cutoffs by $r_c=s/\alpha$ and $k_c=2s\alpha$, respectively. Construct the SOE approximations of $\xi^{\pm}(k,z)$ and $\erf(\alpha z)$ following Section~\ref{subsec::SOEapp}.
		
		\Procedure\textnormal{SOEwald2D}{}   
		\State Sort all the particles according to their $z$ coordinates, as $z_1 < z_2 < \cdots < z_N$.
		\State Compute~$U_{\ell,\text{SOE}}^{\bm{k}}$ for $|\bm{k}|\leq k_c$ as well as $U_{\ell,\text{SOE}}^{\bm{0}}$ according to Section~\ref{subsec::reforEwald2}. 
		\State Compute~$U_{s}$ by direct truncation in real space according to Eq.~\eqref{eq:phi_s} with cutoff $r_c$. 
		\State Compute $U_{\text{self}}$ according to Eqs.~\eqref{eq::self}.
		\State Compute $U = \sum_{|\bm{k}|\leq k_c} U_{\ell, \text{SOE}}^{\bm{k}} + U_{\ell,\text{SOE}}^{\bm{0}} - U_{\text{self}} + U_{s} + U_{\text{p-s}}$.
		\State Compute forces $\bm{F}_{i}$ using a similar procedure as that of $U$.
		\EndProcedure
		
		\State \textbf{Output}: Total electrostatic energy $U$ and forces $\bm{F}_{i}$.
		
	\end{algorithmic}\label{alg:SOEwald2D}
\end{algorithm}

\subsection{Error analysis for the SOEwald2D algorithm}\label{subsec::errSOEwald2D}
Here we derive error estimates for the SOEwald2D summation. 
The total error in the interaction energy $U$ consists of the truncation error and the SOE approximation error:
\begin{equation}
	\mathscr{E}_{U} := \mathscr{E}_{U_{s}}(r_c, \alpha) + \mathscr{E}_{U_{\ell}}(r_c, \alpha) + \sum_{\bm{k}\neq\bm{0}}\mathscr{E}_{U_{\ell},\text{SOE}}^{ \bm{k}} + \mathscr{E}_{U_{\ell},\text{SOE}}^{\bm{0}}\;,
\end{equation}
where the first two terms are the truncation error of Ewald2D summation and have already been provided in Proposition~\ref{prop::2.12}. 
The remainder two terms are the error due to the SOE approximation for the Fourier space components, where
\begin{equation}
	\mathscr{E}_{U_{\ell},\text{SOE}}^{\bm{k}}:=U_{\ell}^{\bm{k}}-U_{\ell,\text{SOE}}^{\bm{k}},\quad \text{and}\quad \mathscr{E}_{U_{\ell},\text{SOE}}^{\bm{0}}:=U_{\ell}^{\bm{0}}- U_{\ell,\text{SOE}}^{\bm{0}}.
\end{equation}
Theorem~\ref{thm:SOE_error} provides  upper bound error estimates, when the Debye-H$\ddot{\text{u}}$ckel (DH) theory is assumed (see~\cite{levin2002electrostatic}, and also~\ref{app::Debye}) to approximate the charge distribution at equilibrium.

\begin{thm}
	\label{thm:SOE_error}
	Given a set of SOE parameters $w_l$ and $s_l$ satisfying Eq.~\eqref{eq::SOE1} and a charge distribution satisfying the DH theory, the SOE approximation error for the Fourier component of interaction energy satisfies:
	\begin{equation}\sum_{\bm{k}\neq\bm{0}}\mathscr{E}_{U_{\ell},\text{SOE}}^{\bm{k}} \leq \frac{2 \lambda_D^2 \alpha^3Q}{\sqrt{\pi}}\varepsilon,
		\quad \text{and} \quad
		\mathscr{E}_{U_{\ell},\text{SOE}}^{\bm{0}} \leq \frac{\sqrt{\pi} \lambda_D^2 (1+2\alpha^2L_z)Q}{\alpha L_xL_y}\varepsilon,
		\label{eq::51}
	\end{equation}
	respectively, where $\lambda_D$ is the Debye length of the Coulomb system. 
\end{thm}

\begin{proof}
	By definitions of $U_{\ell}^{\bm{k}}$ and $U_{\ell,\text{SOE}}^{\bm{k}}$, one has
	\begin{equation}\label{eq::42}
		U_{\ell}^{\bm{k}}-U_{\ell,\text{SOE}}^{\bm{k}}= \frac{\pi}{2 L_xL_y}\sum_{i=1}^{N}\sum_{j\neq i}q_{i}q_{j}\frac{e^{\m{i} \bm{k}\cdot\bm{\rho}_{ij}}}{k}\mathscr{E}_{\xi^{\pm}},
	\end{equation}
	where 
	\begin{equation}\label{eq::44}
		\mathscr{E}_{\xi^{\pm}} := \abs{\xi^{+}(k,z_{ij})-\xi_{M}^{+}(k,z_{ij})} + \abs{\xi^{-}(k,z_{ij})-\xi_{M}^{-}(k,z_{ij})}\;,
	\end{equation}
	By Theorem~\ref{thm:error_xi}, one has
	\begin{equation}
		|\mathscr{E}_{\xi^{\pm}}|\leq\frac{4\alpha e^{-k^2/(4\alpha^2)}}{\sqrt{\pi}k}\varepsilon\;.
	\end{equation}
	Substituting Eq.~\eqref{eq::44} into Eq.~\eqref{eq::42} and using the DH approximation, one gets
	\begin{equation}\label{eq::47}
		\begin{split}   
			\left| \mathscr{E}_{U_{\ell},\text{SOE}}^{\bm{k}} \right| \leq \frac{2\sqrt{\pi} \lambda_D^2 \alpha Q}{L_x L_y} \frac{e^{-k^2 / (4\alpha^2)}}{k^2} \varepsilon.
		\end{split}
	\end{equation}
	To adequately consider the error in Fourier space, the thermodynamic limit is commonly considered~\cite{kolafa1992cutoff,deserno1998mesh}, wherein the sum over wave vectors is replaced by an integral over $\bm{k}$:
	\begin{equation}\label{eq::integral2}
		\sum_{\bm{k}\neq\bm{0}}\approx \frac{L_xL_y}{(2\pi)^2}\int_{\frac{2\pi}{L}}^{\infty}kdk\int_{0}^{2\pi}d\theta,
	\end{equation}
	where $(k,\theta)$ are the polar coordinates and $L=\max\{L_x,L_y\}$. It then follows that
	\begin{equation}
		\sum_{\bm{k}\neq\bm{0}}\mathscr{E}_{U_{\ell},\text{SOE}}^{ \bm{k}}\leq \frac{2\lambda_D^2\alpha^3Q}{\sqrt{\pi}}\varepsilon.
	\end{equation}
	Finally, recalling the SOE approximation errors of $\erf(\cdot)$ and Gaussian functions given by Eqs.~\eqref{eq::erfSOE} and \eqref{eq::SOE1}, one obtains 
	\begin{equation}
		\begin{split}  
			\left|\mathscr{E}_{U_{\ell},\text{SOE}}^{\bm{0}}\right|\leq \frac{\sqrt{\pi} \lambda_D^2 (1+2\alpha^2L_z)Q}{\alpha L_xL_y}\varepsilon\;.
		\end{split}
	\end{equation}
	This finishes the proof.
	
\end{proof}

Based on Proposition~\ref{prop::2.12} and Theorems~\ref{thm:SOE_error}, we conclude that the overall absolute error in $U$ scales as $\mathscr{E}_{U}\sim \mathcal{O}(\varepsilon N)$. 
Notably, for systems sharing the same charge distribution, the Coulomb interaction energy $U \sim \mathcal{O}(N)$. 
Thus we anticipate that our method will maintain a fixed relative error in $U$. 
This will be verified through numerical tests in Section~\ref{subsec::errSOE}.

\subsection{Complexity analysis for the SOEwald2D}
In this section, we analyze the complexity of the SOEwald2D method summarized in Algorithm~\ref{alg:SOEwald2D}. 
The main computational cost is contributed by the following steps: $N$ particle sorting in $z$, the real and reciprocal space summations.
For sorting (Step 4 in Algorithm~\ref{alg:SOEwald2D}), taking advantage of the quasi-2D confinement, various sorting algorithm are suitable, for example, the bucket sorting algorithm~\cite{cormen2022introduction} results in an $\mathcal{O}(N)$ complexity.
To achieve an optimal complexity, the cost of the real and reciprocal space summations (Steps 5 and 6) need to be balanced. 
To analyze it, we first define $\rho_{s}$ and $\rho_{\ell}$ by the average densities in the real and Fourier spaces 
\begin{equation}
	\rho_{s}=\frac{N}{L_x L_y L_z},\quad \text{and} \quad \rho_{\ell}=\frac{L_xL_y}{(2\pi)^2},
\end{equation}
respectively.
The cost $C_s$ for computing the short-range interaction $U_{s}$ scales as
\begin{equation}\label{eq::cs}
	C_{s}=\frac{4\pi}{3}r_c^3\rho_{s}N=\frac{4\pi s^3N^2}{3\alpha^3 L_xL_yL_z}.
\end{equation}
Meanwhile, the total cost $ C_{\ell}$ in computing $U_{\ell,\text{SOE}}^{\bm{k}}$ for all $0<|\bm{k}|\leq k_c$ and $U_{\ell,\text{SOE}}^{\bm{0}}$ is given by
\begin{equation}\label{eq::cll}
	C_{\ell} = \pi k_c^2\rho_{\ell}MN = \frac{1}{\pi}s^2 \alpha^2 L_x L_y M N
\end{equation}
since the recursive computation requires $\mathcal{O}(MN)$ operations for each $\bm{k}$. 
To balance~$C_{s}$ and~$C_{\ell}$, one takes
\begin{equation}
	\alpha \sim \frac{N^{1/5}}{L_x^{2/5}L_y^{2/5}L_z^{1/5}},
\end{equation}
leading to the optimal complexity
\begin{equation}\label{eq::70}
	C_{s} = C_{\ell} \sim \mathcal{O}(N^{7/5}).
\end{equation}
The self interaction~$U_{\text{self}}$ (Step 7) can be directly calculated with a complexity of $\mathcal{O}(N)$, and cost of summing up the total energy (Step 8) is clearly $\mathcal{O}(1)$.
By taking into consideration that the force calculation (Step $9$) requires asymptotically the same cost as the energy calculation (Steps $4$-$8$), it can be concluded that the overall computational complexity of the SOEwald2D algorithm is $\mathcal{O}(N^{7/5})$. 
It is clearly much faster than the original Ewald2D method which scales as $\mathcal{O}(N^{2})$; and surprisingly, it is even slightly faster than Ewald3D for fully-periodic systems, which scales as $\mathcal{O}(N^{3/2})$.

\begin{rmk}\label{rmk::extreme}
	For the extreme case, $L_z \ll \min\{L_x,L_y\}$, the neighboring region for the short-range interaction reduces to a cylinder with radius $r_c$ due to the strong confinement, rather than a spherical region. 
	In this case, one has:
	\begin{equation}
		C_{s}\sim 2\pi r_c^2 \frac{N}{L_x L_y} N=\frac{2\pi s^2N^2}{\alpha^2L_xL_yL_z}.
	\end{equation}
	By simple calculation, the optimal complexity is found to be $\mathcal{O}(N^{3/2})$, which is the same as that of the Ewald3D summation for fully-periodic problems. 
\end{rmk}


\section{Random batch SOEwald2D method} \label{sec:rbm}

In this section, we will introduce a stochastic algorithm designed to accelerate the SOEwald2D method in particle simulations, reducing the complexity to $\mathcal{O}(N)$. 
Unlike existing methods relying on either FFT or FMM-based techniques to reduce the complexity, our idea involves adopting mini-batch stochastic approximation over Fourier modes, with importance sampling for variance reduction. 
More precisely, let us consider the Fourier sum over $\bm k \in \mathcal{K}^2$ for a given kernel $f(\bm{k})$,
one can alternatively understand the Fourier sum as an expectation
\begin{equation}
	\mu:=\sum_{\bm{k}\in\mathcal{K}^2}\frac{f(\bm{k})}{h(\bm{k})}h(\bm{k})=\mathbb{E}_{\bm k\sim h(\bm{k})}\left[\frac{f(\bm{k})}{h(\bm{k})}\right],
\end{equation}
where $\mathbb{E}_{\bm k\sim h(\bm{k})}$ denotes the expectation with $\bm{k}$ sampled from a chosen probability measure $h(\bm{k})$ defined on the lattice $\bm k \in \mathcal{K}^2$. 
Instead of calculating the summation directly or using FFT, a mini-batch of Fourier modes (with batch size $P$) sampled from $h(\bm{k})$ are employed to estimate the expectation, resulting in an efficient stochastic algorithm.

It is worth noting that the random mini-batch strategy originated from stochastic gradient descent~\cite{robbins1951stochastic} in machine learning and was first introduced in the study of interacting particle systems by Jin, {\it et al.}~\cite{jin2020random}, called the random batch method (RBM).
The method was proved to be successful in various areas, including nonconvex optimization~\cite{ghadimi2016mini}, Monte Carlo simulations~\cite{li2020random}, optimal control~\cite{ko2021model}, and quantum simulations~\cite{jin2021randomquantum}. 
Recently, this idea has been applied to fully-periodic Lennard-Jones and Coulomb systems~\cite{liang2021random, jin2021random, liang2023SISC}, demonstrating superscalability in large-scale simulations~\cite{liang2022superscalability, liang2022improved,gao2024rbmdmoleculardynamicspackage}. 
For long-range interactions such as Coulomb, to accurately reproduce the long-range electrostatic correlations, the so-called symmetry-preserving mean-field (SPMF) condition~\cite{hu2014symmetry} has been proposed. The SPMF is originated from the local molecular field theory for Coulomb systems~\cite{chen2006local, hu2010efficient}, which states that algorithms must share a mean-field property, that is
the averaged integration for the computed potential over certain directions should equal that of the exact $1/r$ Coulomb potential.
For fully-periodic systems, by carefully imposing the SPMF in the random batch approximation, it has been shown that the long-range electrostatic correlations can be accurately captured~\cite{gao2023JCTC}.

However, when formulating algorithms for quasi-2D systems, the direct application of the random mini-batch idea introduces formidable challenges.
The classical Ewald2D, either in the closed form (Eq.~\eqref{eq:philk}) or the integral form (Eq.~\eqref{eq::phil}), are unsuitable for random batch sampling:
(1) the closed form demands $\mathcal{O}(N^2)$ complexity even with batch size $P\sim\mathcal O(1)$; 
(2) the integral form is singular at $k=\kappa=0$, giving rise to significant variance.
In this section, we will show that the idea of random mini-batch can now be easily incorporated into the SOEwald2D algorithm based on the reformulation of Ewald2D proposed in Section~\ref{subsec::reforEwald2}, resulting in the Random Batch SOEwald2D~(RBSE2D) method, which can accurately satisfy the SPMF condition for quasi-2D geometry.
Detailed analyses will also be provided.

\subsection{The $\mathcal{O}(N)$ stochastic algorithm} \label{subsec::rbis}


It has been shown in Eqs.~\eqref{eq::pairssum8}, \eqref{eq:varphi_k} and \eqref{eq::SOEphi} that, after applying the SOE approximation to $U_{\ell}^{\bm{k}}$, the Fourier space summation in the SOEwald2D can be compactly written as
\begin{equation}
	\sum_{\bm{k}\neq\bm{0}} U_{\ell,\text{SOE}}^{\bm{k}} = \sum_{\bm{k}\neq\bm{0}} \widetilde{\varphi} (\bm{k})\;,
\end{equation}
where $\widetilde{\varphi} ({\bm{k}})$ is defined as
\begin{equation}
	\widetilde{\varphi} ({\bm{k}}) := e^{-\frac{k^2}{4 \alpha}} \left[ \frac{2\alpha\sqrt{\pi}}{L_xL_y}\sum_{\ell = 1}^{M}  \frac{w_l \sum_{1\leq j < i \leq N}q_{i} q_{j} e^{\m{i}\V{k} \cdot \V{\rho}_{ij}} \left(2\alpha s_l e^{-k z_{ij}}-2k e^{-\alpha s_l z_{ij}}\right) }{k(\alpha^2 s_l^2 - k^2)} \right] \;.
\end{equation}
Comparing to the original Ewald2D formula given in Section~\ref{subsec::elec}, one finds a Gaussian decay factor explicitly, which can be normalized for the purpose of importance sampling.
Thus, we take
\begin{equation}\label{eq::hk}
	h(\bm{k}) := \frac{e^{-k^2/(4\alpha^2)}}{H}\quad\text{with}\quad H := \sum_{\bm{k}\neq\bm{0}}e^{-k^2/(4\alpha^2)},
\end{equation}
where $H$ serves as a normalization factor. 
By the Poisson summation formula (see Lemma~\ref{lem::Poisson}), one has 
\begin{equation}\label{eq::Happ}
	H=\frac{\alpha L_xL_y}{\pi}\sum_{m_x,m_y\in\mathbb{Z}}e^{-\alpha(m_x^2L_x^2+m_y^2L_y^2)}-1,
\end{equation}
where $m_{\xi}=L_{\xi}k_{\xi}/2\pi$ with $\xi\in\{x,y\}$. 
The computation of $H$ is cheap, since Eq.~\eqref{eq::Happ} can be simply truncated to obtain a good approximation. 
Generally speaking, $m_{\xi}=\pm 2$ is enough since $\alpha L_{\xi}\gg 1$.
Then using the Metropolis algorithm~\cite{metropolis1953equation, hastings1970monte} (se \ref{app::Metropolis}), a random mini-batch of frequencies $\{\bm{k}_{\eta}\}_{\eta=1}^P$ is sampled, and the Fourier component of energy can be approximated as:
\begin{equation}\label{eq::RBapp}
	\sum_{\bm{k}\neq\bm{0}} U_{\ell,\text{SOE}}^{\bm{k}} \approx U_{\ell,*}^{\bm{k}\neq\bm{0}} := \frac{H}{P}\sum_{\eta=1}^{P}\widetilde{\varphi}^{\text{RB}}(\bm{k}_{\eta})
\end{equation}
where~$\widetilde{\varphi}^{\text{RB}}({\bm{k}})$ satisfies
\begin{equation}
	\begin{split}
		\widetilde{\varphi}^{\text{RB}}({\bm{k}}) e^{ - \frac{k^2}{4\alpha}} = \widetilde{\varphi}({\bm{k}}) \;,
	\end{split}
\end{equation}
and the corresponding estimator of the force in Fourier space is given by
\begin{equation}
	\sum_{\bm{k}\neq\bm{0}}\bm{F}_{\ell,\text{SOE}}^{\bm{k},i}\approx \bm{F}_{\ell,*}^{\bm{k}\neq\bm{0},i}=-\frac{H}{P} \sum_{\eta=1}^{P}\nabla_{\bm{r}_{i}}\widetilde{\varphi}^{\text{RB}}(\bm{k}_{\eta})\;.
\end{equation}
Each~$\widetilde{\varphi}^{\text{RB}}({\bm{k}})$ and~$\nabla_{\bm{r}_{i}}\widetilde{\varphi}^{\text{RB}}(\bm{k})$ are pairwise summations, fit into the general form of Eq.~\eqref{eq::33}, and can be efficiently computed using the recursive procedure outlined in Eqs.~\eqref{eq::S}-\eqref{eq::35}. 
Due to the use of importance sampling, it is ensured that the aforementioned random estimators are unbiased and have reduced variances, which will be proven in Section~\ref{subsec::consis}.
It is also worth noting that (1) the $\bm k=\bm 0$ mode is excluded in the stochastic approximation and is always computed in an actual MD simulation. Since the averaged potential for the $\bm 0$th mode over the $xy$-plane equal to that of the exact $1/r$ potential, the SPMF condition~\cite{hu2014symmetry} is satisfied (only up to an $\mathcal O(\varepsilon)$ SOE approximation error); (2) for the  $\bm k\neq \bm 0$ modes, random batch sampling is adopted, and it will be justified that $P$ can be chosen independent of $N$; typically, one can choose $P\sim \mathcal O(1)$.

In an actual MD simulation, one will utilize these unbiased estimators along with an appropriate heat bath to complete the particle evolution.  
Except for the summation over Fourier modes $\bm{k}$, the methods of the RBSE2D for other components of both energy and force are the same as those in the SOEwald2D, and the algorithm is outlined in Algorithm~\ref{alg:RBSE2D}.

We now analyze the complexity of the RBSE2D method per time step. 
Similar to the strategy in some FFT-based solvers~\cite{deserno1998mesh,lindbo2012fast}, one may choose $\alpha$ such that the time cost in real space is cheap and the computation in the Fourier space is accelerated. 
More precisely, one chooses
\begin{equation}\label{eq::59}
	\alpha\sim\frac{N^{1/3}}{L_x^{1/3}L_y^{1/3}L_z^{1/3}}
\end{equation}
so that the complexity for the real space part is $C_{s}\sim \mathcal{O}(N)$. 
By using the random batch approximation Eq.~\eqref{eq::RBapp}, the number of frequencies to be considered is then reduced to $\mathcal{O}(P)$ per step, and  the complexity for the Fourier part is $\mathcal{O}(PN)$, even for the challenging cases where the system is untra-thin, i.e., $L_z\ll \min\{L_x,L_y\}$.
These imply that the RBSE2D method has linear complexity per time step if one chooses $P\sim \mathcal{O}(1)$, then by selecting $\alpha$ according to Eq.~\eqref{eq::59}, 
the overall complexity of the RBSE2D is $\mathcal{O}(N)$ for all quasi-2D system setups.

\begin{algorithm}[ht] 
	\caption{The random batch sum-of-exponentials Ewald2D method}
	\begin{algorithmic}[1]
		
		\State \textbf{Input}: Initialize the size of the simulation box $(L_x,L_y, L_z)$, as well as the positions, velocities, and charges of all particles. Choose a precision requirement $\varepsilon$ as well as batch size $P$.
		
		\State \textbf{Precomputation}: Determine Ewald splitting parameters $\alpha$ and $s$ according to Eqs.~\eqref{eq::59} and \eqref{eq:trunction_error}, respectively. Generate real space cutoff by $r_c=s/\alpha$. Construct the SOE approximation of $\xi^{\pm}(k,z)$ and $\erf(\alpha z)$ following Section~\ref{subsec::SOEapp}.
		
		\Procedure\text{RBSE2D}{}       
		\State Draw $P$ frequencies $\{\bm{k}_{\eta}\}_{\eta=1}^{P}$ using the Metropolis algorithm;
		\State Sort all the particles according to their $z$ coordinates, such that $z_1 < z_2 < \cdots < z_N$;
		\State Compute unbiased Fourier space energy $U_{\ell}^{\bm{k}\neq\bm{0},*}$ by importance sampling Eq.~\eqref{eq::RBapp};
		\State Compute SOE-approximated zero-frequency part $U_{\ell,\text{SOE}}^{\bm{0}}$ according to Section~\ref{subsec::reforEwald2}.
		\State Compute $U_{\text{self}}$ and $U_{\text{p-s}}$ via Eqs.~\eqref{eq::self} and \eqref{eq:spectial}, respectively.
		\State Compute~$U_{s}$ by direct truncation in real space via Eq.~\eqref{eq:phi_s} with cutoff $r_c$. 
		\State Compute $U^{*}=U_{\ell}^{\bm{k}\neq\bm{0},*}+U_{\ell,\text{SOE}}^{\bm{0}}-U_{\text{self}}+U_{s}+U_{\text{p-s}}$.
		\State Compute forces $\bm{F}_{i}^{*}$ via a similar procedure as that of $U^{*}$.
		\EndProcedure
		
		\State \textbf{Output}: Unbiased electrostatic energy $U^{*}$ and forces $\bm{F}_{i}^{*}$.
		
	\end{algorithmic}\label{alg:RBSE2D}
\end{algorithm}

\subsection{Consistency and variance analysis}
\label{subsec::consis}

In this section, we provide theoretical analysis for the RBSE2D method. We start with considering the fluctuations, i.e., the stochastic error introduced by the importance sampling at each time step.
The fluctuations for the Fourier space components of the energy and the force acting on the $i$th particle are defined as follows:
\begin{equation}\label{eq::xichi}
	\Xi := \sum_{\bm{k}\neq\bm{0}}\left( U_{\ell}^{\bm{k}}-U_{\ell,*}^{\bm{k}}\right),\quad\text{and}\quad\V{\chi}_{i} :=\sum_{\bm{k}\neq\bm{0}}\left(\bm{F}_{\ell}^{\bm{k},i}-\bm{F}_{\ell,*}^{\bm{k},i}\right)\;.
\end{equation}
Proposition~\ref{prop:unbaised} is obtained directly by the definition of the importance sampling:
\begin{prop}\label{prop:unbaised}
	$U_{\ell,*}^{\bm{k}\neq\bm{0}}$ and $\bm{F}_{\ell,*}^{\bm{k}\neq\bm{0},i}$ are unbiased estimators, i.e. $\mathbb{E}\Xi= 0$, $\mathbb{E} \V{\chi}_{i} = \bm 0$, and their variances can be expressed by
	\begin{equation}\label{eq::Exi}
		\mathbb{E}\Xi^2=\frac{H}{P}\sum_{\bm{k}_1\neq 0}e^{-k_1^2/(4\alpha^2)}\left|\widetilde{\varphi}^{\text{RB}}(\bm{k}_1)-\frac{1}{H} \sum_{\bm{k}_2\neq 0}\widetilde{\varphi}^{\text{RB}}(\bm{k}_2)e^{-k_2^2/(4\alpha^2)}\right|^2
	\end{equation}
	and
	\begin{equation}\label{eq::Echi}
		\mathbb{E}|\V{\chi}_{i}|^2=\frac{H}{P}\sum_{\bm{k}_1\neq 0}e^{-k_1^2/(4\alpha^2)}\left|\nabla_{\bm{r}_i}\widetilde{\varphi}^{\text{RB}}(\bm{k}_1)-\frac{1}{H} \sum_{\bm{k}_2\neq 0}\nabla_{\bm{r}_i}\widetilde{\varphi}^{\text{RB}}(\bm{k}_2)e^{-k_2^2/(4\alpha^2)}\right|^2.
	\end{equation}
\end{prop}

Furthermore, under the Debye-H$\ddot{\text{u}}$ckel approximation, one has the following~Lemma~\ref{lem::upper_bound_phiRB} for the upper bounds of random batch approximations.

\begin{lem} 
	Under the assumption of the DH theory, $|\widetilde{\varphi}^{\emph{RB}}(\bm{k})|$ and $|\nabla_{\bm{r}_i}\widetilde{\varphi}^{\emph{RB}}(\bm{k})|$ have upper bounds
	\begin{equation}
		\abs{ \widetilde{\varphi}^{\emph{RB}}(\bm{k})}\leq\frac{2\sqrt{\pi}\lambda_D^2 Q}{L_xL_y k}\left(\sqrt{\pi}+\frac{\alpha \varepsilon}{k}\right),
		\quad
		\abs{\nabla_{\bm{r}_{i}} \widetilde{\varphi}^{\emph{RB}}(\bm{k})}\leq  \frac{\pi \lambda_D^2 q_{i}^2}{L_xL_y}\left[3+\frac{\alpha}{\sqrt{\pi}}\left(1+\frac{2\sqrt{2}\varepsilon}{k}\right)\right], 
	\end{equation}
	where $\lambda_{D}$ represents the Debye length.
	\label{lem::upper_bound_phiRB}
\end{lem}
\begin{proof}
	By the definition of $\widetilde{\varphi}^{\text{RB}}(\bm{k})$, one has
	\begin{equation}
		\abs{ \widetilde{\varphi}^{\text{RB}}(\bm{k})}\leq \frac{1}{e^{-k^2/(4\alpha^2)}}\left(\left|U_{\ell,\text{SOE}}^{\bm{k}}-U_{\ell}^{\bm{k}}\right|+\left|U_{\ell}^{\bm{k}}\right|\right)\;.
	\end{equation}
	An estimation for the first term is given in Theorem~\ref{thm:SOE_error}. 
	To estimate the second term, one may write Eq.~\eqref{eq::pairssum8} as
	\begin{equation}
		U_{\ell}^{\bm{k}} = \frac{\pi}{2 L_x L_y} \sum_{i, j = 1}^N q_i q_j \frac{e^{\m{i} \V{k} \cdot \V{\rho}_{ij}}}{k} \left[ \xi^{+} (k, z_{ij})+\xi^{-} (k, z_{ij}) \right]\;.
	\end{equation}
	Then using the integral representation of $\xi^{\pm}$ Eq.~\eqref{eq::xi20}, one obtains the following estimate
	\begin{equation}\label{eq::118}
		\begin{split}
			e^{\frac{k^2}{4 \alpha^2}} \left[\xi^{+} (k, z)+\xi^{-} (k, z)\right] & = \frac{2\alpha}{\sqrt{\pi}} \left(e^{k z} \int_{z}^{\infty} e^{-\alpha^2 t^2 - kt} dt+e^{-k z} \int_{-z}^{\infty} e^{-\alpha^2 t^2 - kt} dt\right)\\
			& \leq \frac{4\alpha}{\sqrt{\pi}} \int_{- \infty}^{\infty} e^{-\alpha^2 t^2} dt= 4\;.
		\end{split}
	\end{equation}
	By employing the DH approximation, one has
	\begin{equation}
		\abs{ \widetilde{\varphi}^{\text{RB}}(\bm{k})}  \leq  \frac{2\sqrt{\pi}\lambda_D^2 Q}{L_xL_y k} \left(\sqrt{\pi} + \frac{\alpha \varepsilon}{k}\right)\;.
	\end{equation}
	Similarly, by taking $z$-derivative of the integral form of $\xi^{\pm}$, the following estimate holds:
	\begin{equation}\label{eq::estiZ}
		\begin{split}
			e^{\frac{k^2}{4 \alpha^2}} \partial_z \left[\xi^{+} (k, z)+\xi^{-} (k, z)\right] & = \frac{2\alpha}{\sqrt{\pi}} \partial_z \left( e^{\pm k z} \int_{\pm z}^{\infty} e^{-\alpha^2 t^2 - kt} dt \right) \\
			& \leq \frac{2\alpha}{\sqrt{\pi}} k \left( e^{\pm k z} \int_{- \infty}^{\infty} e^{-\alpha^2 t^2 - kt} dt + e^{-\alpha^2 z^2} \right) \\
			& \leq \left(2 + \frac{2\alpha}{\sqrt{\pi}} \right) k\;.
		\end{split}
	\end{equation}
	Combining Lemma~\ref{lem::forceerr} with Eq.~\eqref{eq::estiZ} and using the DH approximation again give
	\begin{equation}
		\begin{split}
			\abs{ \nabla_{\bm{r}_i}\widetilde{\varphi}^{\text{RB}}(\bm{k})}&\leq \frac{1}{e^{-k^2/(4\alpha^2)}}\left(\left|\nabla_{\bm{r}_i}U_{\ell,\text{SOE}}^{\bm{k}}-\nabla_{\bm{r}_i}U_{\ell}^{\bm{k}}\right|+\left|\nabla_{\bm{r}_i}U_{\ell}^{\bm{k}}\right|\right)\\
			&\leq \frac{\pi \lambda_D^2 q_{i}^2}{L_xL_y}\left[3+\frac{\alpha}{\sqrt{\pi}}\left(1+\frac{2\sqrt{2}\varepsilon}{k}\right)\right].
		\end{split}
	\end{equation}
\end{proof}

Finally, by Lemma~\ref{lem::upper_bound_phiRB}, one has the following Theorem~\ref{thm:unbaised} for the boundness and convergence in the fluctuations originated from the random batch approximation.
\begin{thm}\label{thm:unbaised}
	Under the assumption of the DH theory, further assume that the SOE approximation error $\varepsilon\ll 1$. 
	Then the variances of the estimators of energy and forces have closed upper bounds
	\begin{equation}
		\mathbb{E}\Xi^2\leq \frac{H}{P}\frac{16\pi^{3/2}\lambda_D^4\alpha Q^2}{L_xL_y},\quad \quad \mathbb{E}|\V{\chi}_{i}|^2\leq\frac{H}{P}\frac{4\sqrt{\pi}\alpha^3(3\sqrt{\pi}+\alpha)\lambda_D^4 q_{i}^4}{L_xL_y}.
	\end{equation}
\end{thm}
\begin{proof}
	By Proposition~\ref{prop:unbaised} and the definition of normalization factor $H$, one has
	\begin{equation}
		\begin{split}
			\mathbb{E}\Xi^2
			& = \frac{1}{P}\sum_{\bm{k}_1\neq \bm{0}}\sum_{\bm{k}_2\neq \bm{0}}e^{-(k_1^2+k_2^2)/(4\alpha^2)}\left[\widetilde{\varphi}^{\text{RB}}(\bm{k}_1)-\widetilde{\varphi}^{\text{RB}}(\bm{k}_2)\right]^2\\
			& = \frac{2}{P}\sum_{\bm{k}_1\neq \bm{0}}\sum_{\bm{k}_2\neq \bm{0}}e^{-(k_1^2+k_2^2)/(4\alpha^2)} \widetilde{\varphi}^{\text{RB}}(\bm{k}_1)^2 - \frac{2}{P} \left[ \sum_{\bm{k}\neq \bm{0}} e^{-k^2 /(4\alpha^2)} \widetilde{\varphi}^{\text{RB}}(\bm{k}) \right]^2\\
			& \leq \frac{2H}{P}\sum_{\bm{k}\neq\bm{0}}e^{-k^2/(4\alpha^2)}\left|\widetilde{\varphi}^{\text{RB}}(\bm{k})\right|^2\;.
		\end{split}
	\end{equation}
	Then by using the upper bound of $\left|\widetilde{\varphi}^{\text{RB}}(\bm{k})\right|$ given in Lemma~\ref{lem::upper_bound_phiRB}, one has 
	\begin{equation}
		\begin{split}
			\mathbb{E}\Xi^2&\leq \frac{2\lambda_D^4Q^2H}{\pi PL_xL_y}\int_{\frac{2\pi}{L}}^{\infty}\frac{e^{-k^2/(4\alpha^2)}}{k^2}\left(\sqrt{\pi}+\frac{\alpha \varepsilon}{k}\right)^24\pi k^2dk\\
			&\leq \frac{H}{P}\frac{16\pi^{3/2}\lambda_D^4\alpha Q^2}{L_xL_y},
		\end{split}
	\end{equation}
	where the $\mathcal{O}(\varepsilon)$ and $\mathcal{O}(\varepsilon^2)$ terms are omitted. Analogously, the variance of force can be estimated by
	\begin{equation}\label{eq::focva}
		\begin{split}
			\mathbb{E}|\V{\chi}_{i}|^2&\leq \frac{2H}{P}\sum_{\bm{k}\neq\bm{0}}e^{-k^2/(4\alpha^2)}\left|\nabla_{\bm{r}_i}\widetilde{\varphi}^{\text{RB}}(\bm{k})\right|^2\\
			&\leq \frac{\lambda_D^4Hq_i^4}{2PL_xL_y} \mathlarger{\int}_{\frac{2\pi}{L}}^{\infty}e^{-k^2/(4\alpha^2)}\left[3+\frac{\alpha}{\sqrt{\pi}}\left(1+\frac{2\sqrt{2}\varepsilon}{k}\right)\right]^24\pi k^2dk\\
			&\leq \frac{H}{P}\frac{4\sqrt{\pi}\alpha^3(3\sqrt{\pi}+\alpha)\lambda_D^4 q_{i}^4}{L_xL_y}.
		\end{split}
	\end{equation}
	
	Finally, by definition Eq.~\ref{eq::hk}, $H$ has the following estimate:
	\begin{equation}\label{eq::happrx}
		\begin{split}
			H&=\sum_{\bm{k}\neq\bm{0}}e^{-k^2/(4\alpha^2)}\leq \frac{L_xL_y}{(2\pi)^2}\int_{\frac{2\pi}{L}}^{\infty} e^{-k^2/(4\alpha^2)}4\pi k^2dk\leq \frac{2\alpha^2L_xL_y}{\sqrt{\pi}}.
		\end{split}
	\end{equation}
	Substituting Eq.~\eqref{eq::happrx} into Eq.~\eqref{eq::focva} gives $\mathbb{E}|\V{\chi}_{i}|^2=\mathcal{O}(1/P)$, and Eq.~\eqref{eq::focva} clearly shows the independence of the estimate on the particle number $N$.
\end{proof}


Theorem~\ref{thm:unbaised} has demonstrated that the variance of force scales as $\mathcal{O}(1/P)$, unaffected by the growth of the system size $N$, provided the same particle density $\rho_s$ or Debye length $\lambda_D$. 
This is crucial for its practical usage in MD simulations, where the dynamical evolution typically relies on force calculations rather than energy. 
In the next section, analyses for the strong convergence of the random batch MD will be discussed, which further supports this observation.

\subsection{Strong convergence} \label{subsec::convergence}

In this section, the convergence of the random batch accelerated MD method, the RBSE2D, will be discussed based on the conclusions given in Section~\ref{subsec::consis}.

One first introduces some additional notations.
Let $\Delta t$ be the discretized time step, and $\bm{r}_i$, $m_i$, and $\bm{p}_i$ represent the position, mass, and momentum of the $i$th particle, respectively. 
In each time step of the simulation, the forces (energies) are computed, and the dynamics are subsequently evolved. 
For ease of discussion, let's consider the commonly used NVT ensemble. 
A thermostat is employed to regulate the system's temperature, ensuring that we sample from the correct distribution. Here, one considers the dynamics with Langevin thermostat~\cite{frenkel2023understanding}:
\begin{equation}\label{eq::langevin}
	\begin{split}
		& d \boldsymbol{r}_i = \frac{\boldsymbol{p}_i}{m_i} d t, \\ 
		& d \boldsymbol{p}_i = \left[\boldsymbol{F}_i - \gamma \frac{\boldsymbol{p}_i}{m_i}\right] d t + \sqrt{\frac{2 \gamma}{\beta}} d \boldsymbol{W}_i,
	\end{split}
\end{equation}
where $\boldsymbol{W}_i$ are i.i.d. Wiener processes, $\gamma$ is the reciprocal characteristic time associated with the thermostat.
Let $(\bm{r}_i^*, \bm{p}_i^*)$ be the phase space trajectory to Eq.~\eqref{eq::langevin}, where the exact force $\bm{F}_i$ is replaced by the random batch approximated stochastic force $\bm{F}_i^*=\bm{F}_i-\bm{\chi}_i$. We further suppose that masses~$m_{i}$ for all $i$ are uniformly bounded.
With these notations, the following theorem is introduced.
\begin{thm} (Strong Convergence)
	Suppose for~$\forall i$, the force~$\V{F}_{i}$ is bounded and Lipschitz and~$\mathbb{E} \V{\chi}_{i} = \bm 0$. Under the synchronization coupling assumption that the same initial values as well as the same Wiener process $\bm{W}_i$ are used, then for any~$T > 0$, there exists $C(T) > 0$ such that
	\begin{equation}\label{eq::bound}
		\sup_{t\in[0,T]}\left( \mathbb{E} \left[ \frac{1}{N} \sum_{i = 1}^N \left(\abs{\V{r}_{i} - \V{r}_{i}^*}^2 + \abs{\V{p}_{i} - \V{p}_{i}^*}^2\right) \right] \right)^{1/2} \leq C(T) \sqrt{\Lambda(N) \Delta t}\;,
	\end{equation}
	where~$\Lambda(N)=\|\mathbb{E} \abs{\V{\chi}_{i}}^2\|_{\infty}$ is the upper bound for the variance in the random batch approximated force. In the Debye-H$\ddot{\text{u}}$ckel regime, $\Lambda(N)$ is independent of $N$ (see Theorem~\ref{thm:unbaised}).
	\label{thm:rbm_consist}
\end{thm}

The proof of Theorem~\ref{thm:rbm_consist} is based on previous work~\cite{jin2021convergence,Ye2023IMA} for the original random batch method~\cite{jin2020random}. However, it should be noted that this theorem may fail to be applied to the RBSE2D method due to the singularity of Coulomb kernel at the origin, which violates the required Lipschitz continuity and boundness conditions. 
Additionally, one may concern that the errors introduced by the SOE approximation and Ewald decomposition might disrupt the convergence of the method. 
Rigorous justification of the convergence could still be very challenging and remains open.
Nevertheless, we argue that Theorem~\ref{thm:rbm_consist} may still hold in practice for several reasons: (1) the Lennard-Jones (LJ) potential, commonly used in molecular dynamics simulations, models strong short-range repulsion between particles, which may mitigate the effect of the singularity of Coulomb kernel; (2) the significant variance reduction achieved through the importance sampling technique; and (3) the errors introduced by the Ewald decomposition and SOE approximations can be effectively controlled according to the error estimates. 
Finally, numerical results presented in Section~\ref{sec:md} also validate the effectiveness of random batch method in capturing finite-time structures and dynamic properties, which aligns with the conclusions of Theorem~\ref{thm:rbm_consist}.


The introduced stochastic errors tend to cancel out over time due to the consistent force approximation. This ``law of large numbers'' effect enables the random batch method to perform well in dynamical simulations, despite its single-step error not being as accurate as other deterministic methods. 
For long-time simulations, a uniform-in-time error estimate has been established for the RBM~\cite{Jin2022MMS}, under some stronger force regularity assumptions and suitable contraction conditions. 
We anticipate that this result will also apply to our RBSE2D method for the reasons mentioned above; however, providing a rigorous justification remains challenging and an open question.


\subsection{Further discussions} \label{subsec::conv}

In this section, further discussions about using the RBSE2D method for MD simulations under other thermostats and ensembles are provided.

In practice, the Nos{\'e}-Hoover (NH) thermostat~\cite{hoover1985canonical} is often adopted for the heat bath, instead of the Langevin thermostat. 
The rigorous proof for the convergence of random batch approximated dynamics with the NH thermostat remains open, whereas we expect Theorem~\ref{thm:rbm_consist} still holds. 
This is because the damping factor introduced in the NH allows adaptively dissipating artificial heat~\cite{jones2011adaptive}, while preserving ergodicity and maintaining the desired Gibbs distribution under the NVT ensemble~\cite{herzog2018exponential}.

An interesting topic is whether the RBSE2D preserves the geometric ergodicity, which is crucial for assessing how quickly the distribution converges to the invariant distribution. 
In a recent paper~\cite{jin2023ergodicity}, the authors prove the ergodicity of random batch interacting particle systems for overdamped Langevin dynamics with smooth interacting potentials. 
Though the singularity of Coulomb potential may not be actually reached during the RBSE2D-based MD simulations, a rigorous justification of ergodicity remains a very challenging problem, which will be left open for future explorations.

In line with discussions in~\cite{10.1063/5.0107140}, the RBSE2D-accelerated Langevin and NH dynamics can be extended to the NPT ensemble by incorporating the approximation of the virial tensor. Other well-known integrators, such as Berendsen~\cite{berendsen1984molecular} and Martyna-Tuckerman-Tobias-Klein~\cite{martyna1996explicit}, are also compatible with the RBSE2D. However, extending the RBSE2D to the NVE ensemble poses extra challenge since the Hamiltonian system is disrupted by the random batch sampling, which can be resolved by a modified Newtonian dynamics~\cite{liang2023energy}:
\begin{equation}\label{eq::NVE}
	\begin{split}
		d\V{r}_i&=\frac{\V{p}_i}{m_i}dt,\quad d \bm{p}_i=\left[\V{F}_i-\bm{\chi}_i\right]dt,\\
		dK&=\frac{1}{\gamma}\left[\mathcal H_0-\mathcal H+\Xi\right] dt.
	\end{split}
\end{equation}
Here, $\Xi$ and $\bm{\chi}_i$ represent the fluctuations of energy and force, as defined in Eq.~\eqref{eq::xichi}. $K=\sum |\bm{p}_i|^2/2m_i$ denotes the instantaneous kinetic energy, and $\mathcal H_0$ and $\mathcal H$ represent the Hamiltonian at the initial and current time steps, respectively. The parameter $\gamma$ represents the relaxation time, determining the interval between successive dissipations of artificial heat within the system. An optimal choice for $\gamma$ typically falls in the range of $10\sim 100\Delta t$. It is worth noting that the distributions obtained using Eq.~\eqref{eq::NVE} have a small deviation of $\mathcal{O}(\Delta t^2/P)$ compared to the correct NVE ensemble~\cite{liang2023energy}.

Finally, we discuss the parameter selection for the RBSE2D method. The accuracy is influenced by three key parameters: the parameter $s$, which controls the truncation error of the Ewald summation; the number of exponentials $M$ in the SOE, which governs the SOE approximation error; and the batch size $P$, which affects the variance of the random batch approximation. For a prescribed tolerance $\varepsilon$, $s$ and $M$ can be determined using Eq.~\eqref{eq:trunction_error} and the convergence rate of the SOE method~\cite{gao2021kernelindependent}, respectively. One can pick the proper $s$ and $M$ to ensure that errors from these two components are both at the $O(\varepsilon)$ level. 
The determination of optimal batch size $P$ is system-dependent, relies on performing some numerical tests. Notably, it is empirically observed that a small $P$ is sufficient, thanks to the importance sampling strategy for variance reduction. Numerical results in Section~\ref{subsec::RBSE2D} show that choosing $P=100$ is adequate for coarse-grained electrolytes. 
Regarding the computational complexity, by substituting Eq.~\eqref{eq::59} into Eqs.~\eqref{eq::cs} and \eqref{eq::cll}, one finds that the computational cost of the RBSE2D algorithm for the near-field grows cubically with $s$; while the cost for the far-field grows linearly with $M$ and $P$, and quadratically with $s$.


\section{Numerical results} \label{sec:md}
In this section, numerical results are presented to verify the accuracy and efficiency of the proposed methods. 
The accuracy of the SOEwald2D method is first assessed by comparing it with the original Ewald2D summation. 
This analysis demonstrates the convergence properties of the SOEwald2D method and its ability to maintain a uniformly controlled error bound. 
Subsequently, we employ both the SOEwald2D and RBSE2D methods in MD simulations for three prototypical systems.
These systems include $1:1$ electrolytes confined by charge-neutral or charged slabs, as well as simulations of cation-only solvent confined by negatively charged slabs. 
Finally, the CPU performance of the proposed methods is presented. 
All these calculations demonstrate the attractive features of the new methods.

\subsection{Accuracy of the SOEwald2D method}\label{subsec::errSOE}

In order to verify the convergence of the SOEwald2D method discussed in Section~\ref{subsec::errSOEwald2D}, one considers a system with equal dimensions of $L_x=L_y=L_z=100$, containing randomly distributed $50$ cations and $50$ anions with strengths $q=\pm 1$,  and confined by neutral slabs. 
The original Ewald2D summation (outlined in Section~\ref{subsec::elec}) serves as a reference method.
The Ewald splitting parameter $\alpha$ is fixed as $0.1$ for both the SOEwald2D and Ewald2D, and the cutoffs $r_c$ and $k_c$ are determined by Eq.~\eqref{eq::rckc}. 

The absolute error in electrostatic energy as a function of $s$ is calculated. 
The results are presented in Figure~\ref{fig:error_fixn}(a) for different number of exponentials in the SOE. 
Specifically, $M=4$, $8$ and $16$ correspond to SOE approximation errors $\varepsilon =10^{-4}$, $10^{-8}$, and $10^{-14}$, respectively. 
The convergence behavior depicted in Figure~\ref{fig:error_fixn}(a) is consistent with our theoretical findings, demonstrating both a decaying rate of $O(e^{-s^2}/s^2)$ and a saturated precision of $O(\varepsilon)$ for the SOEwald2D method. 
We also investigate how the relative error in energy varies as the system size scales, while keeping the density $\rho_{s}$ constant. 
The results presented in Figure~\ref{fig:error_fixn}(b) reveal that the error is nearly unaffected by the size of the system, which aligns with the analysis presented in Section~\ref{subsec::errSOEwald2D}.

\begin{figure}[ht]
	\centering
	\includegraphics[width=0.48\textwidth]{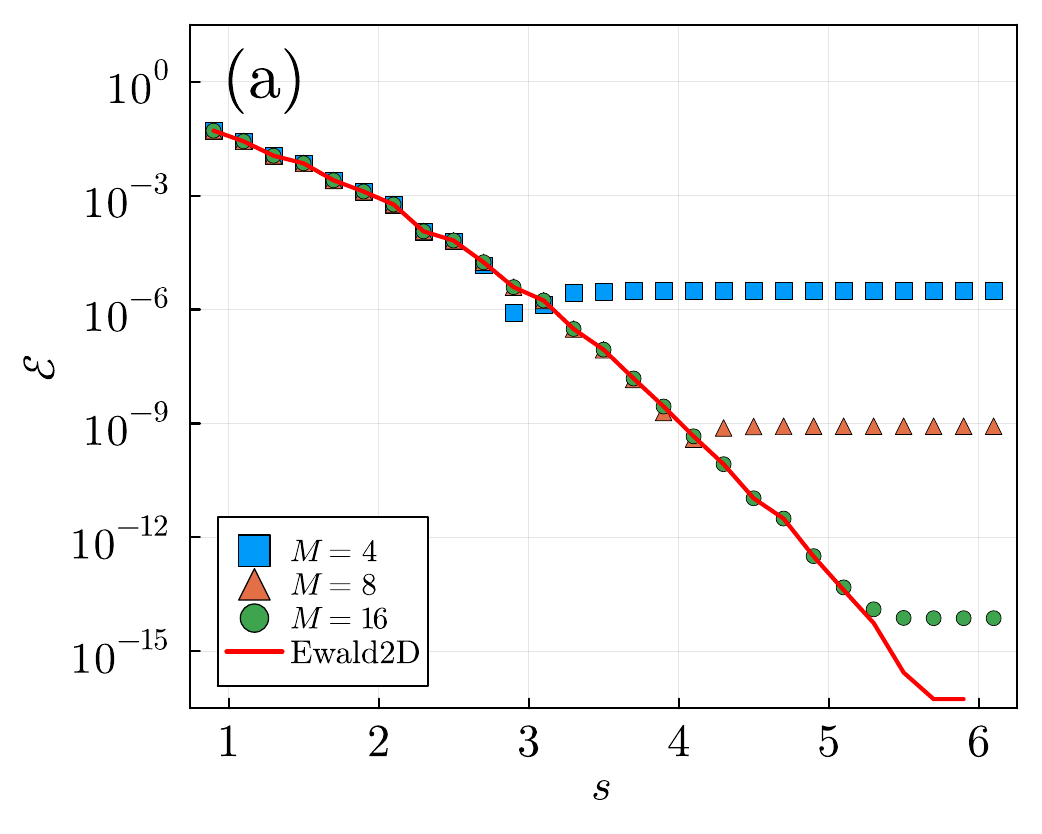}
	\includegraphics[width=0.48\textwidth]{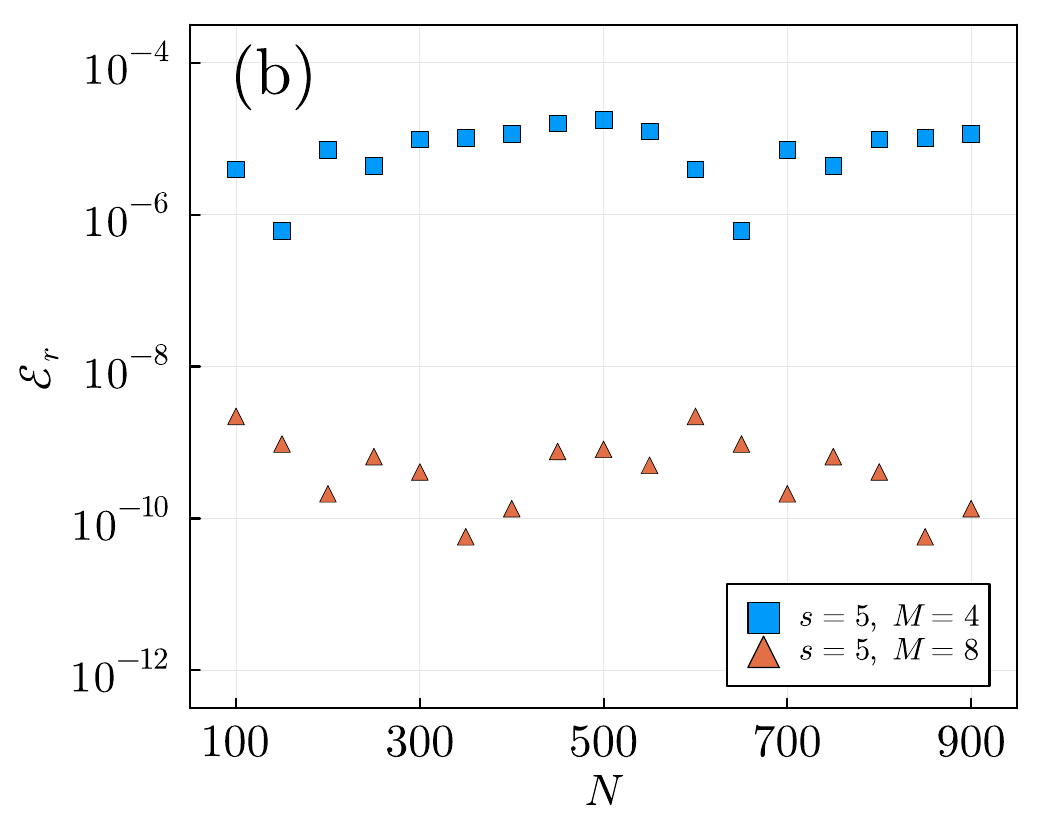}
	\caption{
		Accuracy in the electrostatic energy by the SOEwald2D method. (a): absolute error as a function of $s$; (b): relative error as a function of total number of ions $N$ with fixed ion density $\rho_{s}$. Results with different number of exponentials $M$ are considered. 
	}
	\label{fig:error_fixn}
\end{figure}

\begin{figure}[ht]
	\centering
	\includegraphics[width=0.6\textwidth]{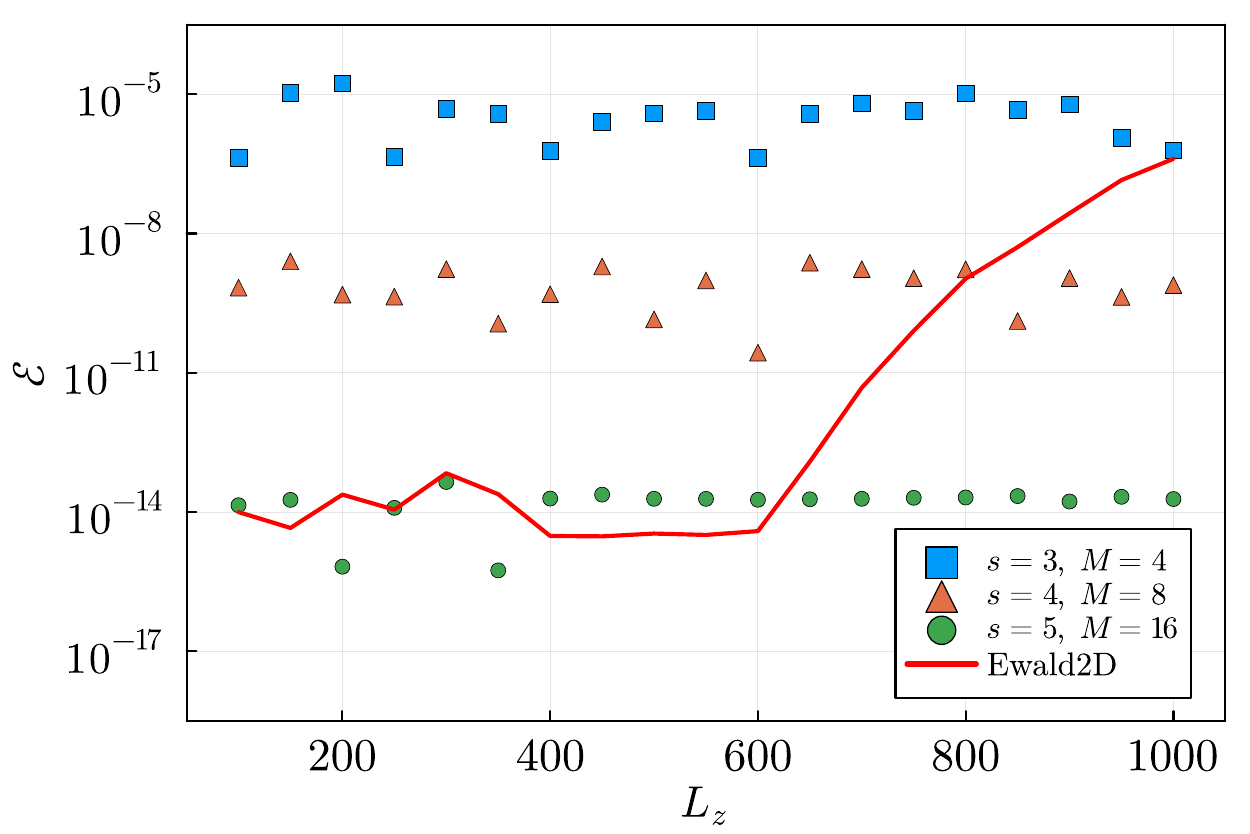} 
	\caption{
		The absolute error in electrostatic energy is evaluated for the SOEwald2D method using three sets of parameters, as well as for the Ewald2D method with $s=5$, as a function of the system's thickness $L_z$. 
	}
	\label{fig:error_Lz}
\end{figure}

As discussed at the end of Section~\ref{subsec::elec}, one notable drawback of the original Ewald2D method is the occurrence of catastrophic error cancellation when the size of the non-periodic dimension increases. 
To quantify this effect, one shall study the absolute error in electrostatic energy as a function of $L_z$. 
The Ewald2D truncation parameter~$s= 3, 4, 5$ are chosen for~$M = 4, 8, 16$, respectively, to obtain optimal accuracy as is guided by Figure~\ref{fig:error_fixn} (a).
The system consists of $100$ uniformly distributed particles, with dimensions $L_x = L_y = 100$ along the periodic dimensions, and the Ewald parameter is set to be $\alpha = 0.1$.
A double-precision floating-point (FP64) arithmetic for both the Ewald2D and SOEwald2D methods is employed, while the reference solution is obtained using the Ewald2D with a quadruple-precision floating-point (FP128) arithmetic, ensuring a sufficient number of significant digits. 
The results presented in Figure~\ref{fig:error_Lz} clearly illustrate that the error of the Ewald2D method increases rapidly with $L_z$.
In contrast, the error of the SOEwald2D method remains independent of $L_z$ for various values of $s$ and $M$, thanks to its stable and well-conditioned summation procedure.

For many existing methods, the accurate evaluation of the forces exert on particles can be strongly influenced by the particle's location in~$z$. 
Due to the uniform convergence of SOE approximation, our method does not suffer from this issue, which is illustrated by two commonly employed examples that have been extensively studied in literature~\cite{lindbo2012fast,de2002electrostatics}. 
In the first example, one considers a system consisting of $50$ anions and $50$ cations arranged in a cubic geometry with a side length of $100$, along with neutral slabs.  
The pointwise error of the force, represented as $\sqrt{\mathscr{E}_{x}^2+\mathscr{E}_{y}^2+\mathscr{E}_{z}^2}$,  is calculated as a function of the particles' $z$-coordinates. 
This evaluation is conducted for various $(s,M)$ pairs, with the Ewald splitting parameter $\alpha=0.1$. 
Figure~\ref{fig:error_ef}(a) clearly demonstrates that the pointwise error in force is independent with its relative position in $z$. 
In the second example, one considers a system of the same size but with two non-neutral slabs. 
The surface charge densities are set as $\sigma_{\mathrm{top}} = \sigma_{\mathrm{bot}} = -0.005$, and the system contains $100$ monovalent cations such that the neutrality condition Eq.~\eqref{eq::chargeneu} is satisfied. 
Figure~\ref{fig:error_ef}(b) indicates that for such non-neutral slabs case, the pointwise error in forces calculated by the SOEwald2D method remains independent with $z$.

\begin{figure}[ht]
	\centering
	\includegraphics[width=\textwidth]{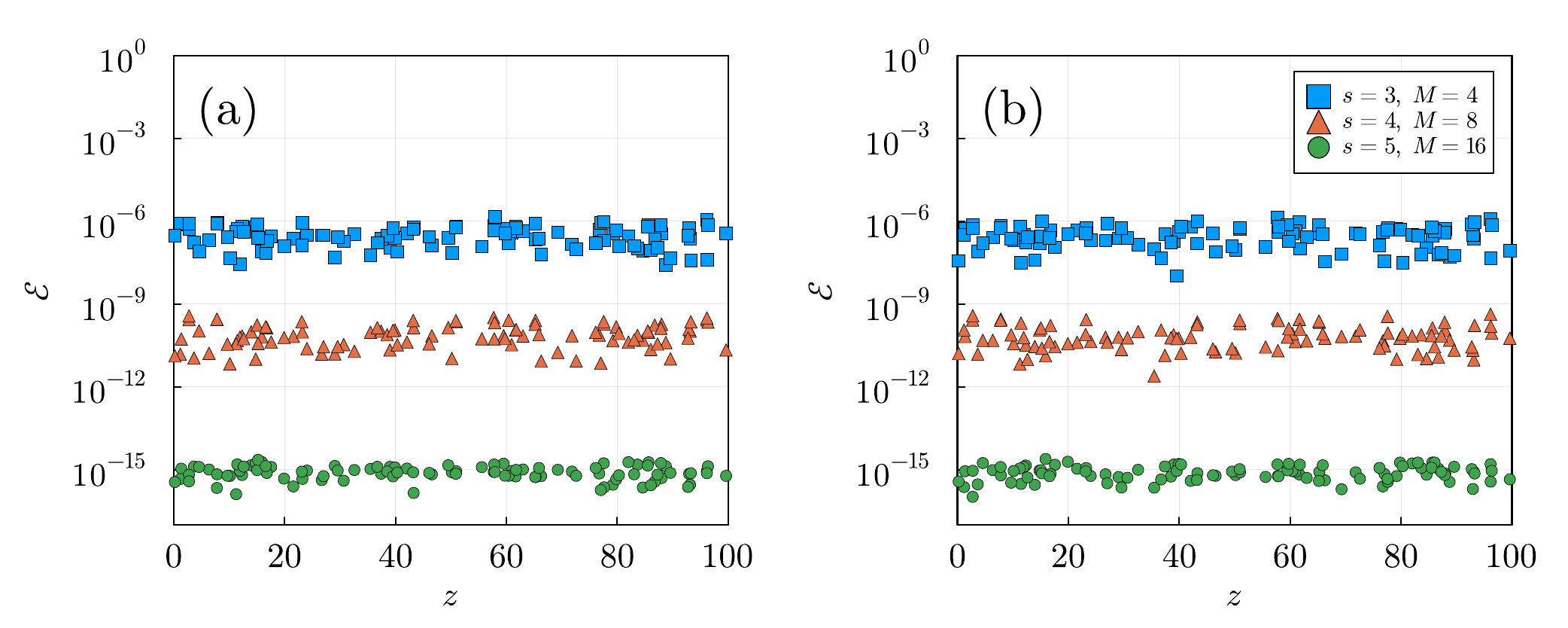} 
	\caption{
		The absolute error in the pointwise electrostatic forces calculated using the SOEwald2D versus particles' $z$-coordinates. 
		Two different scenarios are considered: (a) uniformly distributed 50 anions and 50 cations and (b) uniformly distributed 100 cations with surface charge densities $\sigma_{\mathrm{top}}= \sigma_{\mathrm{bot}} = -0.005$.
	}\label{fig:error_ef}
\end{figure}

\subsection{Accuracy of the RBSE2D method}\label{subsec::RBSE2D}

In contrast to the deterministic SOEwald2D and Ewald2D methods, the  RBSE2D employs unbiased stochastic approximations and its convergence should be investigated in the sense of ensemble averages, as has been carefully discussed in Sec.~\ref{sec:rbm}. 
Therefore, we conduct a series of MD simulations to validate the accuracy of the ensemble averaged equilibrium and dynamical quantities such as particles' concentrations and  mean-squared displacements (MSD) computed using the RBSE2D algorithm.

\begin{figure}[ht]
	\centering
	\begin{minipage}[c]{\textwidth}
		\centering
		\includegraphics[width=0.6\textwidth]{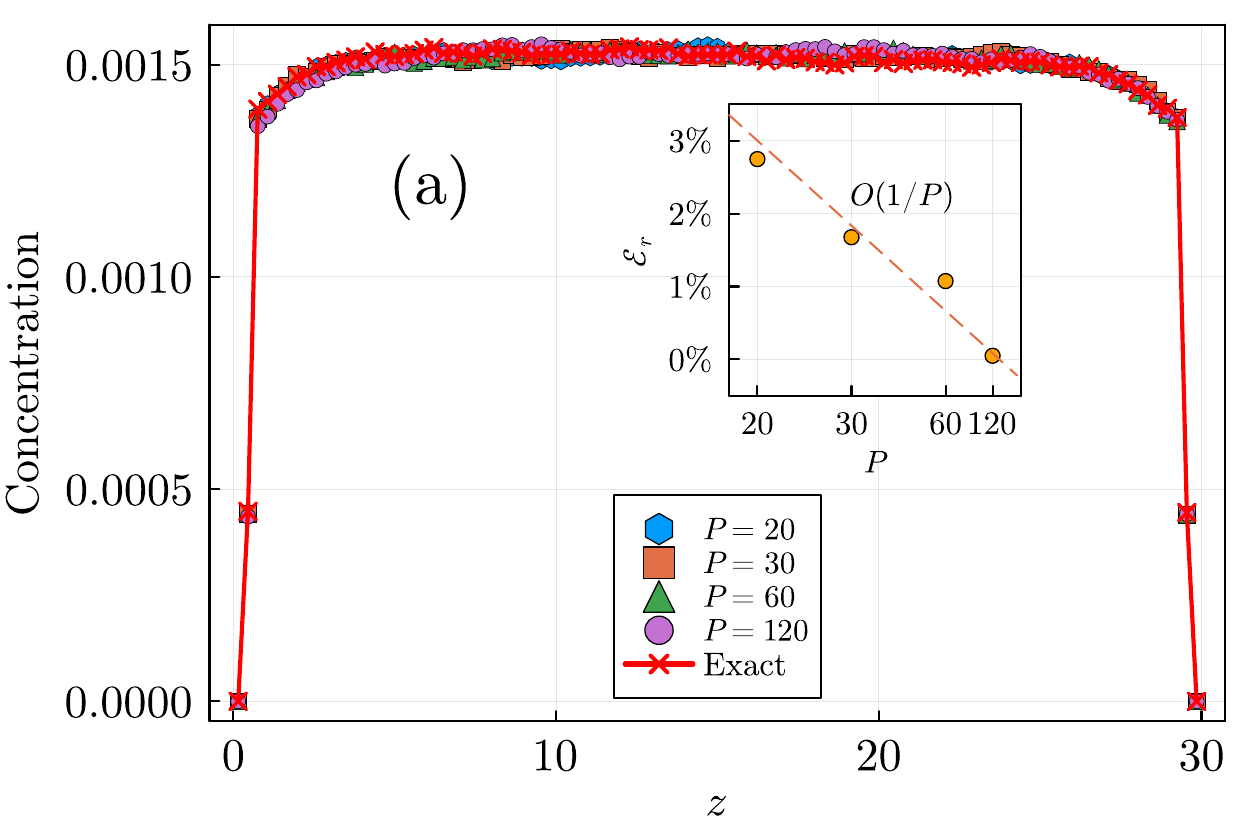}
		\label{fig:compare}
	\end{minipage} \\
	\begin{minipage}[c]{\textwidth}
		\centering
		\includegraphics[width=\textwidth]{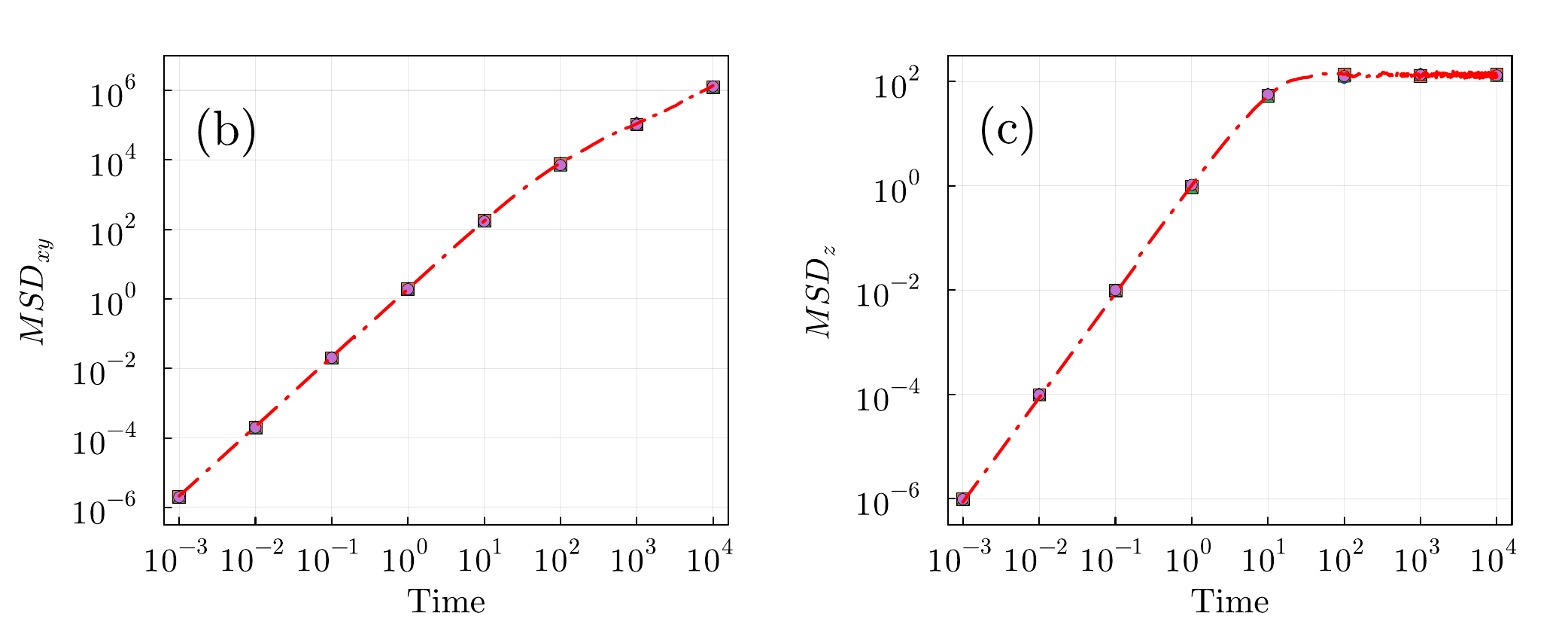} 
		\label{fig:msd_norm}
	\end{minipage} 
	\caption{
		(a) The concentration of cations along $z$, with subplot indicating the convergence in the relative error of the average electrostatic energy as a function of batch size~$P$; (b) and (c) the MSD profiles in~$xy$ and~$z$ against time for a $1:1$ electrolyte confined by neutral slabs. 
		Results by using different batch sizes $P=20, 30, 60, 120$ are shown. 
	}
	\label{fig:norm}
\end{figure}

The first benchmark example is a coarse-grained MD simulation of  $1:1$ electrolytes in the NVT ensemble. 
Following the primitive model~\cite{frenkel2023understanding}, ions are represented as soft spheres with diameter $\sigma$ and mass $m$, interacting through the Coulomb potential and a purely repulsive shifted-truncated Lennard Jones (LJ) potential. 
The LJ potential is given by
\begin{equation}
	U_{\text{LJ}}(r) = 
	\begin{cases}
		4 \epsilon \left[ \left(\dfrac{\sigma}{r}\right)^{12}-\left(\dfrac{\sigma}{r}\right)^6 + \dfrac{1}{4}\right],\quad & r < r_{\text{LJ}}, \\
		0, & r \geq r_{\text{LJ}},
	\end{cases}
\end{equation}
where $r_{\text{LJ}} = 2^{1/6} \sigma$ is the LJ cutoff, $\epsilon = k_B T$ is the coupling strength, $k_B$ is the Boltzmann constant, and $T$ is the external temperature. 
The simulation box has dimensions $L_x = L_y = 100 \sigma$ and $L_z = 30 \sigma$, where the ions confined within the central region by purely repulsive LJ walls located at $z = 0$ and $z = 30 \sigma$ with $\epsilon_{\text{wall}} = \epsilon_{\text{LJ}}$ and $\sigma_{\text{wall}} = 0.5 \sigma$. 
The system contains $218$ cations and anions, and both two walls are neutral. 
The simulation is performed with the time step~$\Delta=0.001\tau$, where~$\tau = \sqrt{m \sigma^2 / \eps_{\text{LJ}}}$ denotes the LJ unit of time. 
The temperature is maintained by using a Nos\'e-Hoover thermostat~\cite{frenkel2023understanding} with relaxation times $0.1\tau$, fluctuating near the reduce external temperature $T=1$.
The system is first equilibrated for~$5 \times 10^5$ steps, and the production phase lasts another~$1 \times 10^7$ steps. 
The configurations are recorded every $100$ steps for statistics. Results produced by the SOEwald2D method with parameters $\alpha=0.1$, $s=4$, and $M=8$ serve as the reference solution, where $\varepsilon\sim 10^{-8}$.

The ion concentration along the $z$-direction is measured, and presented in Figure~\ref{fig:norm}. 
For the RBSE2D method, simulations with varying batch sizes $P$ are performed, while keeping other parameters fixed at $\alpha = 0.3$, $s=4$, and $M=8$. 
It is observed that the results for all choices of $P$ are in excellent agreement with those obtained using the accurate SOEwald2D method. 
Furthermore, one evaluates the MSDs along both the periodic dimensions (Figure~\ref{fig:norm}(a)) and the non-periodic dimension (Figure~\ref{fig:norm}(b)), which describe the particles' anisotropic dynamic properties across a wide range of time scales. 
The RBSE2D methods for all $P$ yield almost identical MSD results as the SOEwald2D method. 
The confinement effect in $z$ leads to a $MSD_z$ profile that clearly indicates a subdiffusion, while $MSD_{xy}$ exhibits a normal diffusion process. 
Clearly, the RBSE2D method successfully captures this anisotropic collective phenomenon.

\begin{figure}[ht]
	\centering
	\includegraphics[width=0.49\textwidth]{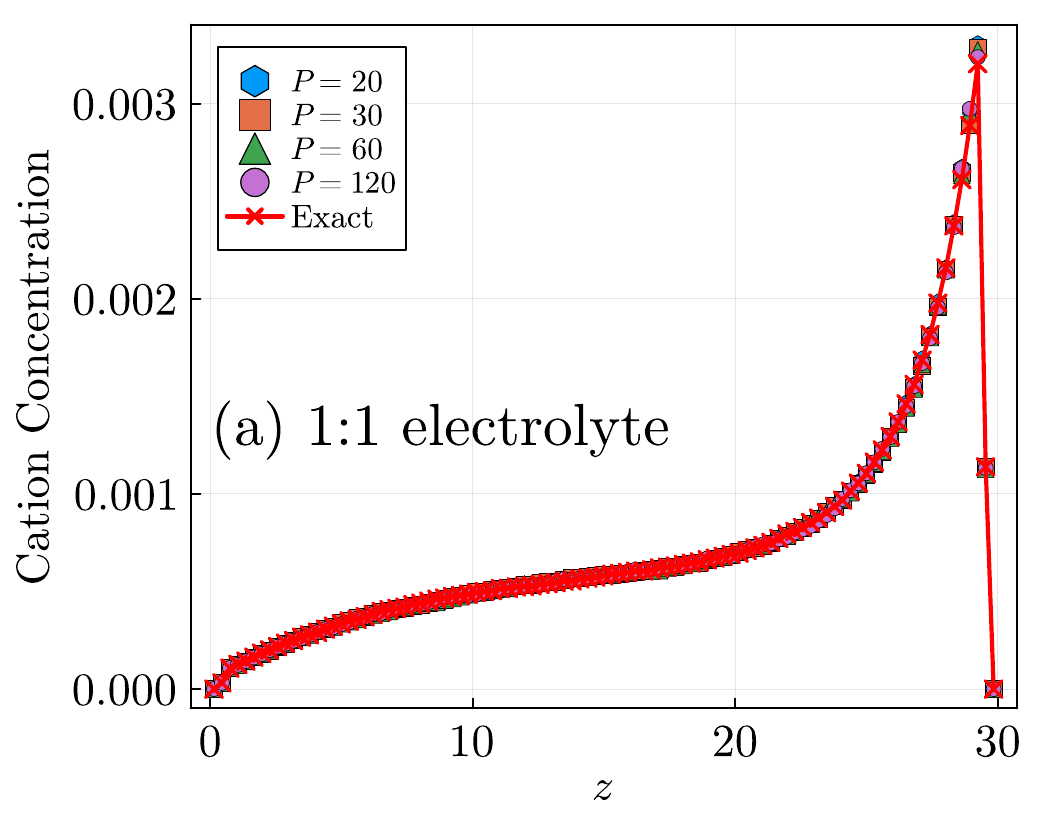} 
	\includegraphics[width=0.49\textwidth]{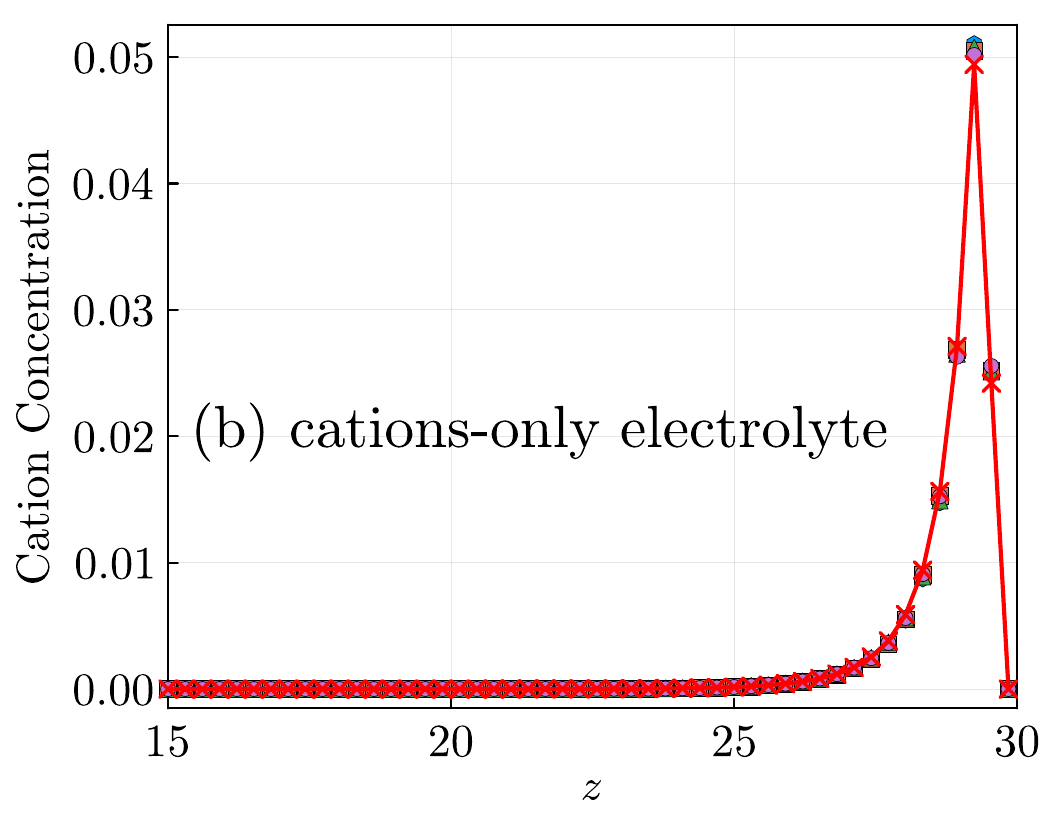}
	\caption{
		Concentration of cations in $z$ for (a) a $1:1$ electrolyte confined between two charged slabs and (b) a cations-only system confined between two slabs, one of which is charged to neutralize the system. 
	}
	\label{fig:Ez_density}
\end{figure}

To assess the performance of our RBSE2D method for systems with non-neutral slabs, one studies
a $1:1$ electrolyte containing~$218$ anions and~$218$ cations with~$q = \pm 1$, and with surface charge densities $\sigma_{\mathrm{bot}}=0.0218$ and~$\sigma_{\mathrm{top}}=-0.0218$. The simulation box is set to be $L_x = L_y = 100 \sigma$ and $L_z = 30 \sigma$.
The resulting equilibrium concentration of cations is shown in Figure~\ref{fig:Ez_density}(a), indicating that results of RBSE2D method with different batch sizes are in good agreement with that of the reference SOEwald2D method. 

We further investigate the most challenging scenario for a system with free cations only, which are confined by non-neutral slabs, so that boundary layers can form at the vicinity of the slabs. 
In particular, the system consists of $436$ monovalent cations and is confined by slabs with surface charge densities $\sigma_{\mathrm{bot}}=0$ and $\sigma_{\mathrm{top}}=-0.0436$ to ensure overall charge neutrality. The concentration of free ions is depicted in Figure~\ref{fig:Ez_density}(b), exhibiting excellent agreement with the results obtained using the SOEwald2D method. 
These findings indicate that choosing a small batch size $P\sim\mathcal O(1)$ is sufficient for generating accurate MD results by using the RBSE2D method.

\subsection{CPU performance}
The CPU performance comparisons among the SOEwald2D, RBSE2D, and the original Ewald2D methods are conducted for MD simulations of $1:1$ electrolyte systems with varying system sizes. 
All calculations are performed on a Linux system equipped with an Intel Xeon Platinum 8358 CPU (2.6 GHz, 1 single core); and by using a self-developed package developed in Julia language. 
To ensure a fair comparison, we maintain the same accuracy across all methods. We fix $s=4$ and set $M=8$ for the SOE approximation, resulting in errors at~$\sim10^{-8}$ for both the Ewald2D and SoEwald2D methods. 
Subsequently, we set the batch size as $P=120$ for the RBSE2D method, with which the RBSE2D-based MD simulations achieve the same accuracy as the SOEwald2D method, as has been illustrated in the previous results.
Finally, for each of the methods, the Ewald splitting parameter $\alpha$ is always adjusted to achieve optimal efficiency.
The CPU time comparison results are summarized in Figure~\ref{fig:times_compare}. 
It is evident that the CPU cost of the Ewald2D, SOEwald2D, and RBSE2D methods scale as $\mathcal{O}(N^2)$, $\mathcal{O}(N^{7/5})$, and $\mathcal{O}(N)$, respectively, which is consistent with our complexity analysis. 
Remarkably, the RBSE2D method demonstrates a significant speedup of $3\times 10^3$-fold compared to the Ewald2D for a system with $N=10^{4}$ particles, enabling large-scale MD simulations on a single core.

An additional observation is regarding the memory consumption and data input/output (I/O) on the maximum system size that can be simulated using the same computational resources. 
In Figure~\ref{fig:times_compare}, it is demonstrated that when utilizing a single CPU core, the Ewald2D and SOEwald2D methods are limited to simulating system sizes of up to about $3 \times 10^4$ and $3 \times 10^5$ particles, respectively. 
In contrast, the RBSOEwald method can handle systems containing about $5 \times 10^6$ particles. 
This is attributed to the reduced number of interacting neighbors that need to be stored in the RBSE2D algorithm, allowing a much smaller real space cutoff $r_c$. 
This significant saving in memory consumption is achieved by the algorithm developed in this study, highlighting its potential as an effective algorithm framework for large-scale simulations of quasi-2D Coulomb systems.

\begin{figure}[ht]
	\centering
	\includegraphics[width=0.6\textwidth]{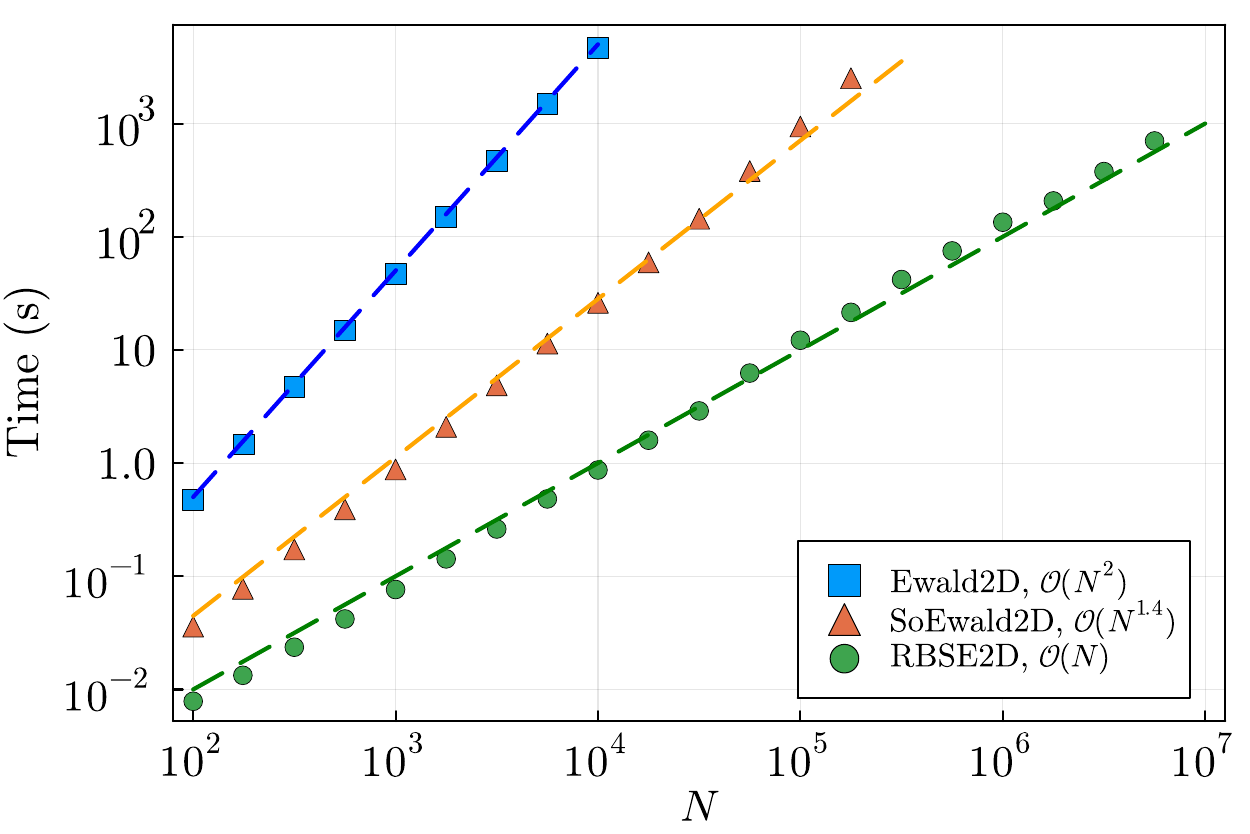} 
	\caption{
		The CPU time cost for the Ewald2D, SOEwald2D, and RBSE2D methods versus the number of particles $N$, with fixed particle density $\rho_s$. 
	}
	\label{fig:times_compare}
\end{figure}

\section{Conclusions}\label{sec:conclusion}
We have proposed the random batch SOEwald2D (RBSE2D) method for MD simulations of doubly periodic systems confined by charged slabs. 
The method utilizes Ewald splitting, and employs the SOE approximation in the non-periodic dimensions and the random batch importance sampling technique for efficient treatment of the Fourier sum in the periodic dimensions. 
Compared to the original Ewald2D summation, the RBSE2D avoids exponential blowup, and reduces the computational complexity from $\mathcal{O}(N^2)$ to $\mathcal{O}(N)$ as well as memory consumption. 
Extensive MD simulations are performed  to   demonstrate the excellent accuracy and performance of the RBSE2D method.

The SOEwald2D can be easily extended to handle other interaction kernels, such as dipolar crystals and Yukawa potentials~\cite{Messina2017PRL,Hou2009PRL}, by utilizing kernel-independent SOE methods such as the VPMR~\cite{gao2021kernelindependent}. 
Our future work will focus on incorporating such techniques into large scale numerical simulations involving long-range interaction kernels. 
We also aim to address the issue of dielectric mismatch, which is important in investigating the interfacial phenomena in electrodes and polyelectrolyte materials~\cite{bagchi2020surface,yuan2020structure}. 
Additionally, we will explore CPU/GPU-based parallelization to further enhance our research.

\section*{Acknowledgement}
Z. G. wish to thank his postdoc mentor Prof. Aleksandar Donev for fruitful discussions on quasi-2D systems, and dedicate this work to him. 
The work of Z. G. and X. G. is supported by the Natural Science Foundation of China (Grant No. 12201146); Natural Science Foundation of Guangdong (Grant No. 2023A1515012197); Basic and applied basic research project of Guangzhou (Grant No. 2023A04J0054); and Guangzhou-HKUST(GZ) Joint research project (Grant Nos. 2023A03J0003 and 2024A03J0606).
The work of J. L. and Z. X. are supported by the Natural Science Foundation of China (grant Nos. 12325113, 12071288 and 12401570) and the Science and Technology Commission of Shanghai Municipality (grant No. 21JC1403700). The work of J. L. is partially supported by the China Postdoctoral Science Foundation (grant No. 2024M751948). The authors also acknowledge the support from the HPC center of Shanghai Jiao Tong University.

\appendix
\setcounter{thm}{0}

\section{Fundamental results from Fourier analysis} \label{app::Fourier}
In this appendix, we state several fundamental results from Fourier analysis for doubly periodic functions, associated with the Fourier transform pair defined in Definition~\ref{Def::Fourier}. These results are useful for us, and their proofs are well established and can be referenced in classical literature, such as in the work of Stein and Shakarchi~\cite{stein2011fourier}.
\begin{lem}\label{lem::Convolution}
	(Convolution theorem) Let $f(\bm{\rho},z)$ and $g(\bm{\rho},z)$ be two functions which are periodic in $\bm{\rho}$ and non-periodic in $z$. Suppose that $f$ and $g$ have Fourier transform $\widetilde{f}$ and $\widetilde{g}$, respectively. Their convolution is defined by
	\begin{equation}\label{eq:Q2D_cov}
		u(\bm{\rho},z):=(f\ast g)(\bm{\rho},z)=\int_{\mathcal{R}^2}\int_{\mathbb{R}}f(\bm{\rho}-\bm{\rho}',z-z')g(\bm{\rho}',z')dz'd\bm{\rho}',
	\end{equation}
	satisfying
	\begin{equation}
		\widetilde{u}(\bm{k},\kappa)=\widetilde{f}(\bm{k},\kappa)\widetilde{g}(\bm{k},\kappa).
	\end{equation}
	
\end{lem}
\begin{lem}\label{lem::Poisson}
	(Poisson summation formula) Let $f(\bm{\rho},z)$ be a function which is periodic in $\bm{\rho}$ and non-periodic in $z$. Suppose that $f$ has Fourier transform $\widetilde{f}$ and $\bm{r}=(\bm{\rho},z)$. Then one has
	\begin{equation}
		\sum_{\bm{m}\in \mathbb{Z}^2} f(\bm{r} + \V{\mathcal{M}}) = \frac{1}{2\pi L_x L_y}\sum_{\bm{k}\in \mathcal{K}^2}\int_{\mathbb{R}}\widetilde{f}(\bm{k},\kappa)e^{\m{i} \bm{k}\cdot\bm{\rho}}e^{\m{i} \kappa z}d\kappa.
	\end{equation}
\end{lem}
\begin{lem}\label{lem::2dfourier}
	(Radially symmetric functions) 
	Suppose that $f(\rho,z)$ is periodic and radially symmetric in $\bm{\rho}$, i.e., $f(\bm{\rho},z)=f(\rho,z)$. Then its Fourier transform $\widetilde{f}$ is also radially symmetric. Indeed, one has
	\begin{equation}
		\widetilde{f}(\rho,z)=2\pi\int_{0}^{\infty}J_0(k\rho)f(\rho,z)\rho d\rho.
	\end{equation}
\end{lem}


\section{Proof of Lemma~\ref{thm::SpectralExpansion}}\label{app::deriv}
By applying the Fourier transform to Poisson's equation
\begin{equation}\label{eq::B.1}
	-\Delta\phi_{\ell}(\bm{\rho},z)=4\pi g(\bm{\rho},z)\ast\tau(\bm{\rho},z),
\end{equation}
one obtains  
\begin{equation}\label{eq::B2}
	\widetilde{\phi}_{\ell}(\bm{k},\kappa)=\frac{4\pi}{k^2+\kappa^2}\widetilde{g}(\bm{k},\kappa)\widetilde{\tau}(\bm{k},\kappa)\quad\text{with}\quad \widetilde{g}(\bm{k},\kappa)=\sum_{j=1}^{N}q_{j} e^{-\m{i} \bm{k}\cdot\bm{\rho}_{j}}e^{-\m{i} \kappa z_{j}}
\end{equation}
via the convolution theorem and the Poisson summation formula (see Lemmas~\ref{lem::Convolution} and \ref{lem::Poisson}, respectively). Applying the inverse Fourier transform to Eq.~\eqref{eq::B2} yields
\begin{equation}\label{eq::phil}
	\phi_{\ell}(\bm{\rho},z)=\frac{2}{L_xL_y}\sum_{j=1}^{N}q_{j}\sum_{\bm{k}\neq\bm{0}}\int_{\mathbb{R}}\frac{e^{-(k^2+\kappa^2)/(4\alpha^2)}}{k^2+\kappa^2}e^{-\m{i} \bm{k}\cdot(\bm{\rho}-\bm{\rho}_{j})}e^{-\m{i} \kappa(z-z_{j})}d\kappa + \phi^{\bm{0}}_{\ell}(z),
\end{equation}
where $\phi^{\bm{0}}_{\ell}(z)$ is the contribution from zero mode. From \cite{oberhettinger2012tables}, one has 
\begin{equation}\label{eq::integral}
	\int_{\mathbb{R}} \frac{e^{-(k^2+\kappa^2)/(4\alpha^2)}}{k^2+\kappa^2} e^{-\m{i} \kappa z} d\kappa = \frac{\pi}{2k} \left[\xi^{+}(k,z)+\xi^{-}(k,z)\right]
\end{equation}
for $\bm{k}\neq\bm{0}$, where $\xi^{\pm}(k,z)$ are defined via Eq.~\eqref{eq::xi20}. Substituting Eq.~\eqref{eq::integral} into the first term of Eq.~\eqref{eq::phil} yields $\phi_{\ell}^{\bm{k}}(\bm{r})$ defined via Eq.~\eqref{eq:philk}.

By Theorem~\ref{order}, the zero-frequency term $\phi^{\bm{0}}_{\ell}(z)$ always exists and its derivation is very subtle. Let us apply the 2D Fourier transform (see Lemma~\ref{lem::2dfourier}) to Poisson's equation Eq.~\eqref{eq::B.1} only on periodic dimensions, and then obtain
\begin{equation}\label{eq::B.5}
	(-\partial_z^2+k^2)\widehat{\phi}_{\ell}(\bm{k},z)=4\pi\widehat{g}(\bm{k},z)\ast_{z}\widehat{\tau}(\bm{k},z),
\end{equation}
where $\ast_z$ indicates the convolution operator along $z$ dimension. Simple calculations suggest
\begin{equation}
	\widehat{g}(\bm{k},z)=\sum_{j=1}^{N}q_{j} e^{-\m{i} \bm{k}\cdot\bm{\rho}_{j}}\delta(z-z_{j}),\quad\text{and}\quad \widehat{\tau}(\bm{k},z)=\frac{\alpha}{\sqrt{\pi}} e^{-k^2/(4\alpha^2)}e^{-\alpha^2z^2}.
\end{equation}
The solution of Eq.~\eqref{eq::B.5} for $\bm{k}=\bm{0}$ can be written as the form of double integral that is only correct up to a linear mode,
\begin{equation}\begin{split}
		\phi_{\ell}^{\bm{0}}(z)&=-\frac{4\pi}{L_xL_y}\int_{-\infty}^{z}\int_{-\infty}^{z_1}\widehat{g}(\bm{0},z_2)\ast_{z_2}\widehat{\tau}(\bm{0},z_2)dz_2dz_1+A_0z+B_0\\
		&=-\frac{2\pi}{L_xL_y}\sum_{j=1}^{N}q_{j}\left[z-z_{j}+(z-z_{j})\erf\left(\alpha(z-z_{j})\right)+\frac{e^{-\alpha^2(z-z_{j})^2}}{\sqrt{\pi}\alpha}\right]+A_0z+B_0.
	\end{split}
\end{equation}
To analyze the short-range component $\phi_{s}(\bm{\rho},z)$ using a procedure similar to Eqs.~\eqref{eq::B.1}-\eqref{eq::integral}, one obtains
\begin{equation}\begin{split}
		\phi_{s}^{\bm{0}}(z) =& \frac{\pi}{L_x L_y} \sum_{j=1}^{N} \lim_{\bm{k}\rightarrow\bm{0}} \frac{1}{k} \left[2e^{-k|z|}-\xi^{+}(\bm{k},z) - \xi^{-}(\bm{k},z)\right]\\
		=& \frac{2\pi}{L_xL_y}\sum_{j=1}^{N}q_{j}\left[-|z-z_{j}|+(z-z_{j})\erf\left(\alpha(z-z_{j})\right)+\frac{e^{-\alpha^2(z-z_{j})^2}}{\sqrt{\pi}\alpha}\right].
	\end{split}
\end{equation}
Since $\phi_{s}^{\bm{0}}(z)+\phi_{\ell}^{\bm{0}}(z)$ matches the boundary condition Eq.~\eqref{eq::boundary1} as $z\rightarrow \pm\infty$ and by the charge neutrality condition, one solves 
\begin{equation}
	A_0 = \frac{2\pi}{L_xL_y}\sum_{j=1}^{N}q_{j}z \equiv 0,\quad \text{and}\quad B_0=-\frac{2\pi}{L_xL_y}\sum_{j=1}^{N}q_{j}z_{j}.
\end{equation}
This result finally gives 
\begin{equation}
	\phi_{\ell}^{\bm{0}}(z)=-\frac{2\pi}{L_xL_y}\sum_{j=1}^{N}q_{j}\left[(z-z_{j})\erf\left(\alpha(z-z_{j})\right)+\frac{e^{-\alpha^2(z-z_{j})^2}}{\sqrt{\pi}\alpha}\right].
\end{equation}

 \section{The ideal-gas assumption for error analysis}\label{app::ideal-gas}
 Let $\bm{\mathcal{\psi}}$ represent a statistical quantity in an interacting particle system, and we aim to analyze its root mean square value given by
 \begin{equation}\label{eq::deltaS}
 	\delta \bm{\mathcal{\psi}}:=\sqrt{\frac{1}{N}\sum_{i=1}^{N}\|\bm{\mathcal{\psi}}_{i}\|^2},
 \end{equation}
 where $\bm{\mathcal{S}}_{i}$ denotes the quantity associated with particle $i$ (e.g., energy for one dimension or force for three dimensions). Assume that $\bm{\mathcal{\psi}}_{i}$ takes the form
\begin{equation}
	\bm{\mathcal{\psi}}_{i}=q_{i} \sum_{j \neq i} q_{j} \bm{\zeta}_{i j},
\end{equation}
due to the superposition principle of particle interactions, which implies that the total effect on particle $i$ can be expressed as the sum of contributions from each $i-j$ pair (including periodic images). Here, $\bm{\zeta}_{i j}$ represents the interaction between two particles. The ideal-gas assumption leads to the following relation
\begin{equation}
	\left\langle\boldsymbol{\zeta}_{i j} \boldsymbol{\zeta}_{i k}\right\rangle=\delta_{j k}\left\langle\boldsymbol{\zeta}_{i j}^2\right\rangle:=\delta_{j k} \zeta^2,
\end{equation}
where the expectation is taken over all particle configurations, and $\zeta$ is a constant. This assumption indicates that any two different particle pairs are uncorrelated, and the variance of each pair is expected to be uniform. In the context of computing the force variance of a charged system, this assumption implies that
\begin{equation}
	\left\langle\|\bm{\mathcal{\psi}}_{i}\|^2\right\rangle=q_{i}^2 \sum_{j, k \neq i} q_{j} q_k\left\langle\boldsymbol{\zeta}_{i j} \boldsymbol{\zeta}_{i k}\right\rangle \approx q_{i}^2 \zeta^2 Q,
\end{equation}
where $Q$ represents the total charge of the system. By applying the law of large numbers, one obtains $\delta\mathcal{\psi}\approx \zeta Q/\sqrt{N}$, which can be utilized for the mean-field estimation of the truncation error.

\section{Proof of Theorem~\ref{thm:ewald2d_phi_error}}\label{app:phierr}
We begin by considering the real space truncation error of electrostatic potential 
\begin{equation}
	\mathscr{E}_{\phi_{s}}(r_c,\alpha)(\bm{r}_{i}) = \sum_{|\bm{r}_{ij} + \V{\mathcal{M}}|>r_c}q_{j} \frac{\erfc(\alpha|\bm{r}_{ij} + \V{\mathcal{M}}|)}{|\bm{r}_{ij} + \V{\mathcal{M}}|}
\end{equation}
for $i$th particle, which involves neglecting interactions beyond $r_c$. By the analysis in \ref{app::ideal-gas}, this part of error can be approximated by $\delta\mathscr{E}_{\phi_{s}}$ with 
\begin{equation}\label{eq::delta^2phi}
	\delta^2\mathscr{E}_{\phi_{s}}=  \frac{1}{V}\sum_{j=1}^{N}q_{j}^2\int_{r_c}^{\infty}\frac{\erfc(\alpha r)^2}{r^2}4\pi r^2dr=\frac{4\pi Q}{V}\mathscr{Q}_{\emph{s}}(\alpha,r_c),
\end{equation}
where $\mathscr{Q}_{\emph{s}}(\alpha,r_c)$ is defined via Eq.~\eqref{eq::Qs2}. Note that the $\erfc(r)$ function satisfies (\cite{olver1997asymptotics}, pp. 109-112)
\begin{equation}\label{eq::asyerfc}
	\erfc(r)=\frac{e^{-r^2}}{\sqrt{\pi}}\sum_{m=0}^{\infty}(-1)^{m}\left(\frac{1}{2}\right)_mz^{-(2m+1)}
\end{equation}
as $r\rightarrow \infty$, where $(x)_m=x(x-1)\cdots(x-m+1)=x!/(x-m)!$ denotes the Pochhammer's symbol. Substituting Eq.~\eqref{eq::asyerfc} into Eq.~\eqref{eq::delta^2phi} and truncating at $m=1$ yields Eq.~\eqref{eq::Qs}. 

The Fourier space error, by~\ref{app::deriv}, is given by
\begin{equation}
	\mathscr{E}_{\phi_{\ell}}(k_c,\alpha)(\bm{r}_{i})=\frac{2}{L_xL_y}\sum_{j=1}^{N}q_{j}\sum_{|\bm{k}|>k_c}\int_{\mathbb{R}}\frac{e^{-(k^2+\kappa^2)/(4\alpha^2)}}{k^2+\kappa^2}e^{-\m{i} \bm{k}\cdot(\bm{\rho}-\bm{\rho}_{j})}e^{-\m{i} \kappa(z-z_{j})}d\kappa.
\end{equation}
For a large $k_c$, one can safely replace the truncation condition with $|\bm{k}+\kappa|>k_c$, resulting in
\begin{equation}
	\begin{split}
		\mathscr{E}_{\phi_{\ell}}(k_c,\alpha)(\bm{r}_{i})&\approx \frac{1}{2\pi^2}\sum_{j=1}^{N}q_{j} \int_{k_c}^{\infty}\int_{-1}^{1}\int_{0}^{2\pi}e^{-k^2/(4\alpha^2)}e^{-i kr_{ij}\cos\varphi}d\theta d\cos\varphi dk\\
		&=\frac{2}{\pi}\int_{k_c}^{\infty}\sum_{j=1}^{N}q_{j}\frac{\sin(kr_{ij})}{kr_{ij}}e^{-k^2/(4\alpha^2)}dk.
	\end{split}
\end{equation}
Here, the summation over Fourier modes is approximated using an integral similar to Eq.~\eqref{eq::integral2}, and one chooses a specific $(k,\theta,\varphi)$ so that the coordinate along $\cos\theta$ of $\bm{k}$ is in the direction of a specific vector $\bm{r}$, and $\bm{k}\cdot\bm{r}=kr \cos\varphi$.
The resulting formula is identical to Eq.~(21) in \cite{kolafa1992cutoff} for the fully-periodic case, and $\delta \mathscr{E}_{\phi_{\ell}}$ can be derived following the approach in \cite{kolafa1992cutoff}.

\section{Force expression of the SOEwald2D} \label{app::force}

The Fourier component of force acting on the $i$th particle can be evaluated by taking the gradient of the energy with respect to the particle's position vector $\bm{r}_{i}$,
\begin{equation}\label{eq::Fi}
	\V{F}_{\ell}^i  \approx \V{F}^{i}_{\text{l},\text{SOE}} = -\grad_{\V{r}_{i}} U_{\ell,\text{SOE}} =  -\sum_{\bm{k}\neq \bm{0}} \grad_{\V{r}_{i}}U_{\ell,\text{SOE}}^{\V{k}} -\grad_{\V{r}_{i}} U_{\ell,\text{SOE}}^{\V{0}}
\end{equation}
where
\begin{align}    
	\grad_{\V{r}_{i}}U_{\ell,\text{SOE}}^{\V{k}} &= - \frac{\pi q_{i}}{L_x L_y} \left[ \sum_{1 \leq j < i} q_{j} \grad_{\V{r}_{i}} \varphi_{\text{SOE}}^{\V{k}}(\V{r}_{i}, \V{r}_{j}) +  \sum_{i < j \leq N} q_{j} \grad_{\V{r}_{i}} \varphi_{\text{SOE}}^{\V{k}}(\V{r}_{j}, \V{r}_{i}) \right]\;,\\
	\grad_{\V{r}_{i}} U_{\ell,\text{SOE}}^{\V{0}} &= - \frac{2 \pi q_{i}}{L_x L_y} \left[ \sum_{1 \leq j < i} q_{j} \grad_{\V{r}_{i}} \varphi^{\bm{0}}_{\text{SOE}}(\bm{r}_{i}, \bm{r}_{j}) + \sum_{i < j \leq N} q_{j} \grad_{\V{r}_{i}} \varphi^{\bm{0}}_{\text{SOE}}(\bm{r}_{j}, \bm{r}_{i})\right]\;.
\end{align}
Using the approximation Eqs.~\eqref{eq::SOEphi},~\eqref{eq::dz_plus} and~\eqref{eq::dz_minus}, one can write the derivative in periodic directions as
\begin{equation}
	\begin{split}
		\partial_{\bm{\rho}_{i}} \varphi_{\text{SOE}}^{\V{k}}(\V{r}_{i}, \V{r}_{j}) & = \frac{\m{i} \V{k} e^{ \m{i} \V{k} \cdot \V{\rho}_{ij}}}{k} \left[\xi^{+}_M(k, z_{ij})+\xi^{-}_M(k, z_{ij})\right]\\
		&=\frac{2\alpha e^{-k^2/(4\alpha^2)}}{\sqrt{\pi}k} \m{i} \bm{k}e^{\m{i} \V{k} \cdot \V{\rho}_{ij}} \sum_{\ell = 1}^{M}  \frac{w_l}{\alpha^2 s_l^2 - k^2}\left( 2 \alpha s_l e^{-k z_{ij}} - 2 k e^{-\alpha s_l z_{ij}}\right),
	\end{split}
\end{equation}
and in~$z$ direction as
\begin{equation}\label{eq:z-der}
	\begin{split}
		\partial_{z_{i}} \varphi_{\text{SOE}}^{\V{k}}(\V{r}_{i}, \V{r}_{j}) & = \frac{e^{\m{i} \V{k} \cdot \V{\rho}_{ij}}}{k} \left[\partial_{z_{i}} \xi^{+}_M(k, z_{ij}) + \partial_{z_{i}} \xi^{-}_M(k, z_{ij})\right]\\
		&=\frac{2 \alpha e^{-k^2/(4\alpha^2)}}{\sqrt{\pi}}e^{ \m{i} \V{k} \cdot \V{\rho}_{ij}} \sum_{\ell = 1}^{M}  \frac{w_l}{\alpha^2 s_l^2 - k^2}\left( - 2 \alpha s_l e^{-k z_{ij}} + 2 \alpha s_l e^{- \alpha s_l z_{ij}}\right)\;.
	\end{split}
\end{equation}
The partial derivatives of zero-frequency mode with respect to the periodic directions are zero, and the SOE approximation of its $z$-derivative is given by
\begin{equation}\label{eq:dzphi_0}
	\begin{split}
		\partial_{z_{i}} \varphi^{\bm{0}}_{\text{SOE}}(\bm{r}_{i},\bm{r}_{j}) 
		& = \sum_{l=1}^{M} \frac{w_l}{\sqrt{\pi}} \partial_{z_{i}} \left[\frac{2z_{ij}}{s_l}+\left(\frac{1}{\alpha} - \frac{2z_{ij}}{s_l}\right)e^{-\alpha s_l z_{ij}}\right] \\
		& = \sum_{l=1}^{M} \frac{w_l}{\sqrt{\pi}} \left[ \frac{2}{s_l} - \left( s_l + \frac{2}{s_l} - 2 \alpha z_{ij} \right) e^{-\alpha s_l z_{ij}} \right]\;.
	\end{split}
\end{equation}

It is important to note that the computation of Fourier space forces using Eq.~\eqref{eq::Fi} follows a common recursive procedure with energy, since it has the same structure as given in Eq.~\eqref{eq::33}, and the overall cost for evaluating force on all~$N$ particles for each~$k$ point also amounts to~$\mathcal{O}(N)$, and the resulting SOEwald2D method is summarized in Algorithm~\ref{alg:SOEwald2D}.

Moreover, Lemma~\ref{lem::forceerr} establishes the overall error on forces $\bm{F}_{i}$, and the proof follows an almost similar approach to what was done for the energy. 

\begin{lem}\label{lem::forceerr}
	The total error of force by the SOEwald2D is given by
	\begin{equation}
		\mathscr{E}_{\bm{F}_{i}} := \mathscr{E}_{\bm{F}_{\emph{s}}^{i}} + \mathscr{E}_{\bm{F}_{\emph{l}}^{i}} + \sum_{\bm{k}\neq \bm{0}} \mathscr{E}_{\bm{F}_{\emph{l}}^i, \emph{SOE}}^{\bm{k}} + \mathscr{E}_{\bm{F}_{\emph{l}}^{i},\emph{SOE}}^{\bm{0}}
	\end{equation}
	where the first two terms are the truncation error and provided in Proposition~\ref{prop::2.12}. The remainder terms 
	\begin{equation}
		\mathscr{E}_{\bm{F}_{\emph{l}}^i,\emph{SOE}}^{\bm{k}} := \bm{F}_{\emph{l}}^{\bm{k}, i} - \bm{F}_{\emph{l},\emph{SOE}}^{\bm{k}, i}, \quad\emph{and} \quad \mathscr{E}_{\bm{F}_{\emph{l}}^{i}, \emph{SOE}}^{\bm{0}} := \bm{F}_{\emph{l}}^{\bm{0}, i}-\bm{F}_{\emph{l}, \emph{SOE}}^{\bm{0}, i}
	\end{equation}
	are the error due to the SOE approximation as Eqs.~\eqref{eq::Fi}-\eqref{eq:z-der}. Given SOE parameters $w_l$ and $s_l$ along with the ideal-gas assumption, one has the following estimate:
	\begin{equation}
		\sum_{\bm{k}\neq\bm{0}} \mathscr{E}_{\bm{F}_{\emph{l}}^i, \emph{SOE}}^{ \bm{k}}\leq \sqrt{2}\lambda_D^2\alpha^2q_{i}^2\varepsilon,\quad\text{and}\quad \mathscr{E}_{\bm{F}_{\emph{l}}^{i},\emph{SOE}}^{\bm{0}}\leq \frac{4\sqrt{\pi}\lambda_D^2 (1+2\alpha)L_z}{L_xL_y}q_{i}^2\varepsilon.
	\end{equation}
\end{lem}

\section{The Debye-H$\ddot{\text{u}}$ckel approximation}\label{app::Debye}
Under the DH approximation, one is able to estimate functions associated with the $i$-th particle in the form:
\begin{equation}
	\mathscr{G}(\bm{r}_i)=\sum_{j\neq i}q_{j}e^{\m{i} \bm{k}\cdot\bm{\rho}_{ij}}f(z_{ij}),
\end{equation}
where $|f(z_{ij})|$ is bounded by a constant $C_f$ independent of $z_{ij}$. The DH theory considers the simplest model of an electrolyte solution confined to the simulation cell, where all $N$ ions are idealized as hard spheres of diameter $r_{a}$ carrying charge $\pm q$ at their centers. The charge neutrality condition requires that $N_+=N_-=N/2$. Let us fix one ion of charge $+q$ at the origin $r=0$ and consider the distribution of the other ions around it.

In the region $0<r\leq r_{a}$, the electrostatic potential $\phi(\bm{r})$ satisfies the Laplace equation $-\Delta\phi(\bm{r})=0$. For $r\geq r_{a}$, the charge density of each species is described by the Boltzmann distribution $\rho_{\pm}(\bm{r})=\pm qe^{\mp\beta q\phi(\bm{r})}\rho_r/2$ with number density $\rho_r=N/V$. In this region, the electrostatic potential satisfies the linearized Poisson-Boltzmann equation~\cite{levin2002electrostatic}:
\begin{equation}
	-\Delta \phi(\bm{r})=2\pi\left[q \rho_r e^{-\beta q\phi(\bm{r})}-q\rho_r e^{+\beta q\phi(\bm{r})}\right]\approx -4\pi \beta q^2\rho_r\phi(\bm{r}),
\end{equation}
and its solution is given by
\begin{equation}
	\phi(\bm{r})=\begin{cases}
		\dfrac{q}{4\pi r}-\dfrac{q\kappa}{4\pi (1+\kappa a)},& r<r_{a},\\[1em]
		\dfrac{qe^{\kappa a}e^{-\kappa r}}{4\pi r(1+\kappa a)},&r\geq r_{a},
	\end{cases}
\end{equation}
where $\kappa=\sqrt{\beta q^2\rho}$ denotes the inverse of Debye length $\lambda_{\text{D}}$. By this definition, the net charge density for $r>r_{a}$ is $\rho_>(\bm{r})=-\kappa^2\phi(\bm{r})$. Let us fix $\bm{r}_i$ at the origin. Given these considerations, for $r\geq r_a$, one obtains the following estimate:
\begin{equation}\label{eq::E.4}
	\begin{split}
		|\mathscr{G}(\bm{r}_i)|&\approx \left|\int_{\mathbb{R}^3\backslash B(\bm{r}_{i}, r_a)}\rho_>(\bm{r})e^{-\m{i} \bm{k}\cdot\bm{\rho}}f(z)d\bm{r}\right|\\
		&\leq \frac{q_iC_fe^{\kappa a}}{4\pi(1+\kappa a)}\int_{a}^{\infty}\frac{e^{-\kappa r}}{r}4\pi r^2dr\\
		&=q_iC_f\lambda_{\text{D}}^2.
	\end{split}
\end{equation}

It is remarked that upper bound Eq.~\eqref{eq::E.4} is derived under the continuum approximation. In the presence of surface charges, the charge distribution along the $z$-direction may lack spatial uniformity. However, due to the confinement of particle distribution between two parallel plates, the integral in Eq.~\eqref{eq::E.4} along the $z$-direction remains bounded. An upper bound in the form of $|\mathscr{G}(\bm{r}_i)|\leq C_s C_{f}q_i$ can still be expected, where $C_s$ is a constant related to the thermodynamic properties of the system.



\section{The Metropolis algorithm} \label{app::Metropolis}

In practice, the Metropolis algorithm~\cite{metropolis1953equation, hastings1970monte} is employed to generate a sequence $\{\bm{k}_{\eta}\}_{\eta=1}^{P}$ from $h(\bm{k})$. 
Since $\bm{k}\circ \bm{L}=2\pi \bm{m}$ with $\bm{m}$ an integer vector, one can conveniently sample from the discrete distribution $\mathcal{H}(\bm{m})=h(\bm{k})$ to equivalently generate $\bm{k}$. 
Once the current state of the Markov chain $\bm{m}_{\eta}=\bm{m}^{\text{old}}$ is known, the algorithm generates a random variable $\bm{m}^*$ with $m_{\xi}^*\sim \mathcal{N}[0,(\alpha L_{\xi})^2/2\pi^2]$, which is the normal distribution with mean zero and variance $(\alpha L_{\xi})^2/2\pi^2$. The new proposal is taken as $\bm{m}^{\text{new}}=\text{round}(m_{x}^*,m_{y}^*)$. To determine the acceptance rate, one obtains the proposal probability 
\begin{equation}
	q(\bm{m}^{\text{new}}|\bm{m}^{\text{old}})=\prod_{\xi\in\{x,y\}}q(m^{\text{new}}_{\xi}|m^{\text{old}}_{\xi})
\end{equation}
where
\begin{equation}
	\begin{split}
		q(m^{\text{new}}_{\xi}|m^{\text{old}}_{\xi})&=\sqrt{\frac{\pi}{(\alpha L_{\xi})^2}}\int_{m^{\text{new}}_{\xi}-\frac{1}{2}}^{m^{\text{new}}_{\xi}+\frac{1}{2}}e^{-\pi^2t^2/(\alpha L_{\xi})^2}dt\\[1em]
		&=\begin{cases}
			\erf\left(\dfrac{\pi}{2\alpha L_{\xi}}\right),&m_{\xi}^{\text{new}}=0,\\[1.15em]
			\dfrac{1}{2}\left[\erf\left(\dfrac{\pi(2|m_{\xi}^{\text{new}}|+1)}{2\alpha L_{\xi}}\right)-\erf\left(\dfrac{\pi(2|m_{\xi}^{\text{new}}|-1)}{2\alpha L_{\xi}}\right)\right],&m_{\xi}^{\text{new}}\neq 0.
		\end{cases}
	\end{split}
\end{equation}
It is worth noting that the proposal distribution $q(\bm{m}^{\text{new}}|\bm{m}^{\text{old}})$ in the Metropolis algorithm presented here does not depend on the current state $\bm{m}^{\text{old}}$. The Metropolis acceptance probability is computed using the formula:
\begin{equation}
	a(\bm{m}^{\text{new}}|\bm{m}^{\text{old}}):=\min\left\{\frac{\mathcal{H}(\bm{m}^{\text{new}})q(\bm{m}^{\text{old}}|\bm{m}^{\text{new}})}{\mathcal{H}(\bm{m}^{\text{old}})q(\bm{m}^{\text{new}}|\bm{m}^{\text{old}})},1\right\}.
\end{equation}
If the proposal is rejected, then $\bm{m}_{\eta+1}=\bm{m}_{\eta}$. If $\bm{m}^{\text{new}}$ is accepted, then $\bm{m}_{\eta+1}=\bm{m}^{\text{new}}$. The sampling procedure has a small error since $\mathcal{H}(\bm{m}^{\text{new}})\approx q(\bm{m}^{\text{new}}|\bm{m}^{\text{old}})$. Our numerical experiments show an average acceptance rate of over $90\%$. Additionally, one can set an integer downsampling rate $\mathscr{D}$, where only one sample is taken from every $\mathscr{D}$ samples, to reduce the correlation between batches in the Metropolis process.

\end{document}